\documentclass[reqno,11pt]{amsart}

\usepackage{amsmath,amsfonts,amssymb,amsthm,epsfig}
\usepackage{mathtools,enumerate}
\usepackage[english]{babel}
\usepackage{tikz}
\usetikzlibrary{arrows}

 \allowdisplaybreaks
\numberwithin{equation}{section}

\font\script=rsfs10 at 11pt
\def\eps{\varepsilon}
\def\H{{\mbox{\script H}\,\,}}

\def\C{\mathcal C}
\def\Cper{C_{\rm per}}
\def\E{\mathcal E}
\def\F{\mathcal F}
\def\G{\mathcal G}
\def\R{\mathbb R}
\def\S{\mathbb S}
\def\SS{\mathcal S}
\def\N{\mathbb N}

\def\bal{\begin{aligned}}
\def\eal{\end{aligned}}
\def\proofof#1{\begin{proof}[Proof of #1]}

\def\step#1#2{\par\noindent{\underline{\it Step~#1.}}\emph{ #2}\\}

\def\XXint#1#2#3{{\setbox0=\hbox{$#1{#2#3}{\int}$} \vcenter{\vspace{-1pt}\hbox{$#2#3$}}\kern-.5\wd0}}

\def\comp{\subset\subset}

\newcommand{\res}{\mathop{\hbox{\vrule height 7pt width .5pt depth 0pt \vrule height .5pt width 6pt depth 0pt}}\nolimits}

\pgfarrowsdeclare{<<<}{>>>}
{
\arrowsize=0.2pt
\advance\arrowsize by .5\pgflinewidth
\pgfarrowsleftextend{-4\arrowsize-.5\pgflinewidth}
\pgfarrowsrightextend{.5\pgflinewidth}
}
{
\arrowsize=0.2pt
\advance\arrowsize by .5\pgflinewidth
\pgfpathmoveto{\pgfpointorigin}
\pgfpathlineto{\pgfpoint{-1.2mm}{-.8mm}}
\pgfusepathqstroke
\pgfpathmoveto{\pgfpointorigin}
\pgfpathlineto{\pgfpoint{-1.2mm}{.8mm}}
\pgfusepathqstroke
}

\newcounter{mt}

\def\maintheoremdeclaration#1{\stepcounter{mt}\newcounter{#1}\setcounter{#1}{\arabic{mt}}}

\maintheoremdeclaration{main}

\newtheorem{theorem}{Theorem}[section]
\newtheorem{lemma}[theorem]{Lemma}

\newtheorem{corol}[theorem]{Corollary}
\newtheorem{defin}[theorem]{Definition}
\newtheorem{remark}[theorem]{Remark}

\begin{document}

\title{The $\eps-\eps^\beta$ property for clusters with double density}

\author{Vincenzo Scattaglia}
\author{A. Pratelli}

\maketitle

\begin{abstract}
This article is devoted to extend the ``$\eps-\eps^\beta$ property'' to the case of clusters in an Euclidean space with a double density.
\end{abstract}

\section{Introduction}

A fundamental tool in the study of isoperimetric problems is a property that we can call ``$\eps-\eps$ property'' of a set $E$. It basically says that it is possible to modify the volume of $E$ of a small quantity $\eps$, at most increasing the perimeter of a quantity $C|\eps|$. It is simple to imagine how can this property be useful. In fact, to show good features of isoperimetric sets (for instance regularity or boundedness of the boundary), one usually notices that if one of these good features fails then it is possible to slightly modify the set, changing the volume by a small quantity $\eps$ and decreasing the perimeter much more than $|\eps|$. One has then to build a competitor, and the ``$\eps-\eps$ property'' is exactly what one needs in order to readjust the volume without wasting the whole gain in the perimeter. As a very simple example, a planar set with a corner cannot be isoperimetric since ``cutting the edge'' for a length $r\ll 1$ changes the volume of order $r^2$ and decreases the perimeter of order $r$. This property is so crucial that some version of it is used in most of the papers regarding isoperimetric problems.\par

This property is obviously valid for every ${\rm C}^1$ set $E$ with the standard Euclidean volume and perimeter. On the contrary, it may easily fail when one considers weighted volume and perimeter, a case which has been classically studied since decades but which gained a particular interest in recent years. The definition of weighted volume and perimeter is very simple; namely, one is given two so-called \emph{densities}, i.e. two l.s.c. functions $f:\R^N\to (0,+\infty)$ and $g:\R^N\times\S^{N-1}\to (0,+\infty)$, and volume and perimeter of a set $E\subseteq\R^N$ of finite perimeter are given by
\begin{align}\label{defvolper}
|E| = \int_E f(x)\, d\H^N(x)\,, && P(E) = \int_{\partial^* E} g(x,\nu_E(x))\, d\H^{N-1}(x)\,,
\end{align}
where $\partial^* E$ is the reduced boundary of $E$ and $\nu_E(x)$ is the outer normal of $E$ at $x\in\partial^*E$. The standard Euclidean case corresponds then to the choice $f\equiv g\equiv 1$. As said above, for general densities the $\eps-\eps$ property can fail. However, one can notice that for a bad set it is usually possible to decrease the perimeter at a higher order than the change in volume, for instance in the example above of a planar set with a corner the gain in perimeter was of order $\sqrt{|\eps|}$, being $\eps$ the volume variation. More in general, in $\R^N$ usually the gain in perimeter is at least of order $|\eps|^{\frac{N-1}N}$. As a consequence, to obtain competitors one does not really need the $\eps-\eps$ property, it is sufficient to have an ``$\eps-\eps^\beta$ property'', that is, one can modify the volume of a small quantity $\eps$ at most increasing the perimeter of $C|\eps|^\beta$ (the formal definition will be given later, in Definition~\ref{defpropeb}). Thanks to what we have just said, it is clear that one would like to have $\beta>\frac{N-1}N$, or at least $\beta=\frac{N-1}N$ with an arbitrarily small constant $C$, so to adjust the volume with still some gain in perimeter.\par

The study of the validity of the $\eps-\eps^\beta$ property for general densities has been started in~\cite{CP}. There, the case of a ``single density'' was considered, that is, the case when $g(x,\nu)=f(x)$ for every $x\in\R^N,\, \nu\in\S^{N-1}$. This is a sort of ``intermediate case'' between the standard Euclidean case and the general case of a double density. This intermediate case is in fact the one considered in several papers dealing with isoperimetric problems with density, but the general case with double density is much more complete, and it appears quite naturally while studying, for instance, Riemannian manifolds with density (see for instance the book~\cite{Morganbook}). In the paper~\cite{CP} it was proved that, as soon as the density $g$ is $\alpha$-H\"older, then the $\eps-\eps^\beta$ property holds, with some $\beta=\beta(\alpha,N)$, given in~(\ref{defbeta}). The value of $\beta$ ranges from $\frac{N-1}N$, in the limit when $\alpha\searrow 0$, and $1$, when $\alpha=1$, so in particular the full $\eps-\eps$ property holds as soon as $g$ is Lipschitz. In the limit case $\alpha=0$, if the density is continuous then the $\eps-\eps^{\frac{N-1}N}$ property holds with an arbitrarily small constant $C$, and as said above this is enough for most purposes.\par

The result was later extended in~\cite{PS19} to the case of a double density. The result is the same, in particular the exponent $\beta$ is still given by~(\ref{defbeta}), and the construction follows the same scheme, but there are some technical complications to solve. The reason for the troubles is not the fact that $f\neq g$, since estimates about volume and about perimeter are actually always distinct, but the fact that $g$ can depend also on the direction and not only on the point.\par

In this paper, we establish the $\eps-\eps^\beta$ property in the case of clusters with a double density. The study of clusters is an important issue among the isoperimetric problems, extensively studied since the pioneering works~\cite{Tay,HMRR,Rei}. A ``cluster'' is simply a finite collection $\E=\{E_1,\, E_2,\, \dots\,,\, E_m\}$ of pairwise disjoint sets in $\R^N$. In the classical Euclidean setting, its ``volume'' and ``perimeter'' are respectively given by
\begin{align*}
\Big( \H^N(E_1),\, \H^N(E_2),\, \dots\,,\, \H^N(E_m) \Big) \in (\R^+)^m \,, && \H^{N-1} \Big( \bigcup\nolimits_{i=1}^m \partial^* E_i\Big)\,.
\end{align*}
In the general setting of a double density that we are considering, it is clear that the volume has to be considered in the sense of~(\ref{defvolper}), that is, we set
\begin{equation}\label{volcl}
\big| \E \big| = \big(|E_1|,\, |E_2|,\, \dots\, ,\, |E_m|\big) \in (\R^m)^+\,.
\end{equation}
Concerning the perimeter, its definition cannot solely depend on the union $\partial^*\E := \cup \partial^* E_i$ of the boundaries. Indeed, in this union the information about the direction of the normal vector at any point is implicitly contained, but only up to a multiplication times $\pm 1$, and this is not enough since $g(x,-\nu)$ does not necessarily coincide with $g(x,\nu)$. It is simple to observe that the correct generalisation of the notion of the perimeter is given by
\begin{equation}\label{percl}
P(\E) = \frac{\sum_{i=1}^m P(E_i) + P\Big(\bigcup\nolimits_{i=1}^m E_i \Big)}2\,,
\end{equation}
which clearly coincides with
\[
\int_{\partial^*\E} g(z)\, d\H^{N-1}(z)
\]
in the simplified case when the function $g$ only depends on the first variable. We are now in position to present the claim of our main result, which is the validity of the $\eps-\eps^\beta$ property in the isoperimetric problem with clusters, extending the result of~\cite{CP,PS19}.

\begin{theorem}\label{main}
Let us assume that $f:\R^N\to (0,+\infty)$ and $g:\R^N\times \S^{N-1}\to (0,+\infty)$ are two l.s.c. and locally bounded functions, and that $g$ is locally $\alpha$-H\"older in the first variable for some $0\leq \alpha\leq 1$. Let moreover $\E$ be a $m$-cluster of finite perimeter in $\R^N$. Then, the $\eps-\eps^\beta$ property holds for $\E$, where $\beta=\beta(\alpha,N)$ is given by
\begin{equation}\label{defbeta}
\beta=\frac{\alpha+(N-1)(1-\alpha)}{\alpha + N(1-\alpha)}\,.
\end{equation}
In addition, if $\alpha=0$ (thus $\beta=\frac{N-1}N$) and $g$ is continuous in the first variable, then the constant $\Cper$ in Definition~\ref{defpropeb} can be taken arbitrarily small, in the sense of Remark~\ref{remCsmall}.
\end{theorem}

The precise study of how the value of the constant $\Cper$ can be bounded in terms of $g$ when $\alpha=0$ and $g$ is continuous is performed in Remark~\ref{bestC}. As said above, the above result is already known, with the same value of $\beta$ given by~(\ref{defbeta}), when one considers sets instead of clusters. In other words, the following result is known, and it has been proved in~\cite{CP} for the special case of the single density, i.e., $g(x,\nu)=f(x)$, and in~\cite{PS19} for the general case.

\begin{theorem}\label{ThPS}
The result of Theorem~\ref{main} is true for single sets, that is, if $m=1$.
\end{theorem}

%
%
As already anticipated in the beginning of this introduction, the $\eps-\eps^\beta$ property is a technical fundamental tool in order to prove properties of isoperimetric sets such as boundedness and regularity. It turns out that the same holds also in the case of isoperimetric clusters for weighted volume and perimeter, with a very similar proof. More precisely, we have the following.
\begin{corol}\label{toadd}
Let $\E$ be a minimal $m$-cluster in $\R^N$, assume that $f$ and $g$ are bounded from above and below (that is, there exists $M>0$ such that for every $x\in\R^N$ and $\nu\in\S^{N-1}$ one has $M^{-1}\leq f(x),\, g(x,\nu)\leq M$, and assume one of the following holds:
\begin{itemize}
\item the cluster $\E$ satisfies the $\eps-\eps^\beta$ property with $\beta > \frac{N-1}N$ or $\beta = \frac{N-1}N$ and constant $C_{\text{per}}$ arbitrarily small;
\item the perimeter density $g$ is converging at infinity to a finite limit $a>0$ uniformly with respect to the second variable;
\end{itemize} 
Then $\E$ is bounded.
\end{corol}

The plan of the paper is very simple. In Section~\ref{sectec} below we collect some standard properties of the sets of finite perimeter, together with the formal definition of the $\eps-\eps^\beta$ property for clusters and with some notation that will be used through the paper. Then, in Section~\ref{sec:a0}, we study the preliminary case $\alpha=0$, which will be obtained by means of the fundamental Lemma~\ref{lemmaa=0}. The next step, in Section~\ref{sec:infil}, consists in presenting a version of the ``infiltration lemma'' which is valid in this context, and which follows by the case $\alpha=0$. And finally, in Section~\ref{sec:proof}, we will use the Infiltration Lemma to conclude the proof of Theorem~\ref{main} for a general $0\leq\alpha\leq 1$.

\subsection{Basic notation and properties of sets of finite perimeter and of clusters\label{sectec}}

This short section is devoted to present some notation that will be used through the paper, and some known results about sets of finite perimeter. First of all, given $m\in\N$, $m\geq 1$, we call \emph{$m$-cluster}, or shortly a \emph{cluster}, a collection $\E=(E_1,\, E_2,\, \dots\,,\,E_m)$ of essentially pairwise disjoint sets in $\R^N$. We say that $\E$ is a \emph{cluster of finite perimeter} if each $E_i$ is a set of finite perimeter (we assume a basic knowledge of the theory of sets of finite perimeter, a reader may refer to~\cite{AFP} for any needed detail). For brevity of notation, we will write
\begin{align*}
E_0=\R^N\setminus \Big( \cup_{i=1}^m E_i\Big)\,, && \partial^*\E = \cup_{i=1}^m \partial^* E_i\,.
\end{align*}

As usual, for any $x\in\R^N$ and any $r>0$ we will denote by $B(x,r)$ the open ball with center in $x$ and radius $r$. Given any set $F\in\R^N$ and any $z\in\R^{N-1}$, we denote by $F_z$ the \emph{vertical section of $F$ relative to $z$}, that is, the $1$-dimensional set $F_z=\{t\in\R:\, (z,t)\in F\}$. Analogously, for any $t\in\R$, we denote by $F^t$ the \emph{horizontal section of $F$ relative to $t$}, that is, the $(N-1)$-dimensional set $F^t= \{ z\in\R^{N-1}:\, (z,t)\in F\}$. A well-known and fundamental result about sections of sets of finite perimeter is the following one, due to Vol'pert. A proof can be found in~\cite{Volpert}, see also~\cite[Theorem~3.108]{AFP}.
\begin{theorem}[Vol'pert Theorem]\label{volpert}
Let $F$ be a set of finite perimeter. Then, for $\H^{N-1}$-a.e. $z\in\R^{N-1}$ the section $F_z$ is a set of finite perimeter in $\R$, and $(\partial F)_z=\partial (F_z)$. Analogously, for $\H^1$-a.e. $t\in\R$ the section $F^t$ is a set of finite perimeter in $\R^{N-1}$, and $(\partial^* F)^t = \partial^* (F^t)$ up to $\H^{N-2}$-negligible subsets.
\end{theorem}

Another key fact about sets of finite perimeter is the following, classical blow-up result, due to De Giorgi, whose proof can be found in~\cite[Theorem~3.59]{AFP}.

\begin{theorem}[Blow-up Theorem]\label{thblowup}
Let $F\subseteq \R^N$ be a set of finite perimeter, and let $x\in\partial^* F$, with exterior normal vector $\nu=\nu_F(x)$. For every $r>0$, define the ``blow-up set'' $F_r = \frac 1r\, (F-x)$, and define the measure $\mu_r = D \chi_{F_r}$ and the half-space $H=\{x\in\R^N:\, x\cdot \nu<0\}$. Then, the sets $F_r$ converge to $H$ in the $L^1_{\rm loc}$ sense, and the measures $\mu_r$ and $|\mu_r|$ converge respectively to $\nu \H^{N-1} \res \partial H$ and to $\H^{N-1}\res \partial H$ in the weak* sense.
\end{theorem}

We conclude this section by presenting the definition of the $\eps-\eps^\beta$ property for clusters, which is the obvious extension of the analogous property for sets.

\begin{defin}[The $\eps-\eps^\beta$ property for clusters]\label{defpropeb}
Let $\E$ be a cluster of finite perimeter, and let $0<\beta\leq 1$. We say that \emph{the $\eps-\eps^\beta$ property for $\E$ holds} if there are three positive constants $\overline r,\, \bar \eps$ and $\Cper$ such that the following holds. Let $\eps\in\R^m$ be such that $|\eps|<\bar\eps$, and let $w\in\R^N$. Then, there exists another cluster $\F$ such that the symmetric difference $\E\Delta\F$ is contained in at most $m(m+1)/2+1$ balls of radius at most $\overline r$, not intersecting $B(w,\overline r)$, and such that
\begin{align}\label{defepsbeta}
\big|\F\big| = \big|\E\big| + \eps\,, && P(\F) \leq P(\E) + \Cper |\eps|^\beta\,,
\end{align}
where volume and perimeter of clusters are given by~(\ref{volcl}) and~(\ref{percl}).
\end{defin}

\begin{remark}\label{remCsmall}
While the exact value of the constant $\Cper$ is usually not so interesting, in some cases, in particular when $\alpha=0$ and so $\beta=\frac{N-1}N$, it is important to have estimates about it, and more specifically it is crucial to have $\Cper\ll 1$. As a consequence, we sometimes write $\Cper[t]$, for any $0<t\leq \bar\eps$, the best constant for which~(\ref{defepsbeta}) holds whenever $|\eps|\leq t$. In our main Theorem~\ref{main} we show that, when $\alpha=0$ and $g$ is continuous in the first variable, then $\Cper$ can be taken arbitrarily small, that is, $\lim_{t\searrow 0} \Cper[t]=0$. An accurate estimate of the rate of convergence, in terms of the modulus of continuity of $g$, will be given in Remark~\ref{bestC}.
\end{remark}

\section{First step: the case $\alpha=0$\label{sec:a0}}

In this section, we start the construction for our proof. Our first step consists in proving the special case $\alpha=0$. This is not only a particular case of our main result, in fact it will be needed in order to give the proof also for $\alpha>0$. More precisely, a crucial tool in the general proof will be an Infiltration Lemma, that we will be able to prove by using the case $\alpha=0$. First of all, we prove the basic estimate which is needed to get the $\eps-\eps^\beta$ property. The result is a generalization of~\cite[Theorem~A]{PS19} to the case of clusters, and our proof follows the same scheme of the proof there, with many complications due to the fact that we are now dealing with clusters.
\begin{defin}\label{defomegaxt}
Given a point $x\in\R^N$, we define
\[
\omega_x(t) = \sup \Big\{\big|g(y,\nu)-g(z,\nu)\big|:\, \nu\in\S^{N-1},\, y,\,z\in B(x,t) \Big\}\,,
\]
and $\omega_x = \lim_{t\searrow 0} \omega_x(t)$. Moreover, we call $M_x\geq 1$ the smallest constant such that $1/M_x\leq f(y)\leq M_x$ and $1/M_x\leq g(y,\nu)\leq M_x$ for every $y\in B(x,1)$ and every $\nu\in\S^{N-1}$.
\end{defin}

\begin{lemma}\label{lemmaa=0}
Let $\E$ be an $m$-cluster, let $i\neq j \in \{0,\,1,\, \dots\, ,\, m\}$ be two indices, and let $\bar x\in \partial^* E_i\cap \partial^* E_j$. Assume that, possibly up to swap $i$ and $j$,
\begin{align}\label{0ornot0}
\hbox{\rm either} \quad i=0,\, && \hbox{\rm or} \quad \H^{N-1}\big(\partial^* E_i \cap \partial^* E_0\big)=0\,.
\end{align}
Then, there exists $C=C(M_{\bar x},N,m)>0$ such that, for every $K> C\sqrt[N]{\omega_{\bar x}}$ and every $r>0$, there is $\bar\eps>0$ so that, for each $0<\eps<\bar\eps$, there is a cluster $\F$ satisfying
\begin{align}\label{oldone}
\E\Delta \F\comp B(\bar x,r)\,, && 
|F_n| = \left\{
\begin{array}{cc}
|E_n|+\eps &\hbox{if $n=i$}\,,\\
|E_n|-\eps &\hbox{if $n=j$}\,,\\
|E_n| &\hbox{if $n\notin \{i,j\}$}\,,
\end{array}\right.
&& P(\F) \leq P(\E) + K \eps^{\frac{N-1}N}\,.
\end{align}
\end{lemma}
\begin{proof}
For simplicity of notations, we assume that $\bar x=0$ and that $\nu_{E_i}(\bar x)=(0,1)\in \R^{N-1}\times \R$. Notice that, since $\bar x\in\partial^*E_i\cap\partial^* E_j$, then necessarily $\nu_{E_j}(\bar x)=(0,-1)$. Since $\bar x$ is fixed, we write for brevity $M$ in place of $M_{\bar x}$. We let $C=C(M, N,m)$ be a large constant, to be precised later, we fix $K>C \sqrt[N]{\omega_{\bar x}}$ and $r>0$, and we let $\bar \eps=\bar \eps(M,N,m,K,r)>0$ be a small constant, also to be found later. We divide the proof in a few steps for clarity.

\step{I}{Choice of the reference cube and reduction to the validity of~(\ref{wecanassume}).}
We start by introducing a small parameter $\rho=\rho(M,N,m,K)$, whose value will be precised later. For a small $a<r$, we call $Q=(-a/2,a/2)^{N-1}\subseteq\R^{N-1}$ and $Q_N=(-a/2,a/2)^N\subseteq\R^N$. Denoting the generic point of $\R^N=\R^{N-1}\times\R$ as $x=(x',x_N)$, the blow-up Theorem~\ref{thblowup} readily ensures that, if $a$ is small enough, 
\begin{gather}
1-\rho \leq \frac{\H^{N-1}(\partial^* E_i \cap Q^N \cap \{-a\rho < x_N < a\rho\})}{a^{N-1}} \leq 1+\rho\,, \label{eq:1} \\
1-\rho \leq \frac{\H^{N-1}(\partial^* E_j \cap Q^N \cap \{-a\rho < x_N < a\rho\})}{a^{N-1}} \leq 1+\rho\,, \label{eq:1bis} \\
\H^{N-1}\big((\partial^* E_i\cup\partial^* E_j) \cap Q^N \setminus \{-a\rho < x_N < a\rho\}\big) \leq \rho \,a^{N-1}\,, \label{eq:2} \\
\H^N\Big( \Big(\big(\{x_N<0\}\setminus E_i\big)\cup \big(\{ x_N>0\}\setminus E_j\big)\Big) \cap Q^N\Big) < \rho^3 \,a^N\,.\label{eq:3}
\end{gather}
As a consequence, a simple integration argument ensures that, possibly replacing $a$ with some $a/2\leq \tilde a\leq a$, we can also assume that
\begin{equation}\label{eq:5}
\H^{N-1}\big(E_n\cap \partial Q^N\big)\leq C_1\rho^3 a^{N-1}\qquad \forall\, n\notin \{i,\,j\} \,,
\end{equation}
where $C_1$ is a geometric constant, only depending on $N$, whose exact value, though elementary to calculate, is not needed in the following. Let now $n\notin \{0,\,i,\,j\}$, and let us assume that
\begin{equation}\label{assstep1}
\H^{N-1}(\partial^* E_n\cap Q^N) \geq \frac{\rho}{3mM^2}\, a^{N-1}\,.
\end{equation}
Let us also call $\Omega=E_n\cap Q^N$ for brevity. Since $\H^{N-1}$-a.e. point of $\partial^* E_n$ also belongs to $\partial^* E_\ell$ for some $\ell\neq n$, then one of the following two cases occur,
\begin{gather}
\H^{N-1}(\partial^* E_n \cap\partial^* E_0\cap Q^N) \geq \bigg(1-\frac 1{3M^2}\bigg) \, \H^{N-1}(\partial^* E_n\cap Q^N) \,,\tag{\ref{assstep1}a}\label{case1}\\
\H^{N-1}(\partial^* E_n\cap\partial^* E_\ell \cap Q^N) \geq \frac 1{3m M^2} \, \H^{N-1}(\partial^* E_n\cap Q^N)\quad \hbox{for some $\ell\notin \{0,\, n\}$}\,.\tag{\ref{assstep1}b}\label{case2}
\end{gather}
In the first case, we can define a modified cluster $\E'$ by putting $E_k'=E_k$ for every $k\notin \{0,\,n\}$, while $E_n'=E_n\setminus \Omega$ and $E_0'=E_0\cup \Omega$. It is simple to compare $P(\E)$ and $P(\E')$, indeed the only difference is that points of $\partial^* \Omega\cap \partial^* E_0$ belong to $\partial^*\E$ and not to $\partial^*\E'$, while points of $\partial^* \Omega\setminus \partial^* E_0$ belong to $\partial^*\E'$ and they might not belong to $\partial^*\E$, or belong to $\partial^*\E$ but with a different cost if $g$ is not symmetric. By~(\ref{case1}), (\ref{eq:5}) and~(\ref{assstep1}) we have then
\[\begin{split}
P(\E)-P(\E') &\geq \frac 1M\, \H^{N-1}(\partial^* \Omega\cap \partial^* E_0) - M \H^{N-1}(\partial^* \Omega \setminus \partial^* E_0)\\
&\geq \frac 1M\,\bigg(1-\frac 1{3M^2}\bigg) \, \H^{N-1}(\partial^* E_n\cap Q^N)\\
&\qquad - M \bigg(\frac 1{3M^2}\, \H^{N-1}(\partial^* E_n\cap Q^N)+ \H^{N-1}(E_n \cap \partial Q^N)\bigg)>\frac \rho{10mM^3}\, a^{N-1}\,,
\end{split}\]
as soon as $\rho$ is small enough. Since $\big||\E|-|\E'|\big| \leq 2 M \H^N(\Omega)<2M\rho^3 a^N$ by~(\ref{eq:3}) then, up to possibly further decrease $\rho$, so to get $\rho^3\ll \rho$, we can easily define another cluster $\E''$ with $\E''\Delta \E \comp B(x,r)$, $|\E''|=|\E|$, and $P(\E'')< P(\E)$. And then, the validity of the lemma in this case is obvious for a sufficiently small $\bar\eps$, even with the last property of~(\ref{oldone}) replaced by $P(\F)<P(\E)$. Thus, the proof is already concluded if~(\ref{case1}) occurs.\par
A similar argument can be done if~(\ref{case2}) occurs. Indeed, this time we define $\E'$ by putting $E_k'=E_k$ for every $k\notin\{\ell,\, n\}$, while $E_n'=E_n\setminus \Omega$ and $E_\ell'=E_\ell\cup \Omega$. Also this time we can easily compare $P(\E)$ with $P(\E')$. Indeed, points of $\partial^* \Omega\cap Q^N\setminus \partial^* E_\ell$ belong both to $\partial^*\E$ and to $\partial^*\E'$, and the relative cost is the same, since both $n,\, \ell\neq 0$. Instead, points of $\partial \Omega^* \cap \partial^* E_\ell$ belong to $\partial^*\E$ and not to $\partial^*\E'$. Therefore, also using~(\ref{case2}), (\ref{eq:5}) and~(\ref{assstep1}), again up to possibly decrease $\rho$ we have
\[\begin{split}
P(\E)-P(\E') &\geq \frac 1M\, \H^{N-1}(\partial^* \Omega\cap \partial^* E_\ell) - M \H^{N-1}(E_n \cap \partial Q^N)\\
&\geq \frac 1{3mM^3} \, \H^{N-1}(\partial^* E_n\cap Q^N)- M \H^{N-1}(E_n \cap \partial Q^N)>\frac \rho{10m^2M^5}\, a^{N-1}\,,
\end{split}\]
so we concude exactly as before. Summarizing, we have already obtained the conclusion both if~(\ref{case1}) holds, and if~(\ref{case2}) holds, and then also if~(\ref{assstep1}) holds. As a consequence, from now on we can assume that for no $n\notin\{0,\,i,\,i\}$ the property~(\ref{assstep1}) holds, that is,
\begin{equation}\label{wecanassume}
\H^{N-1}(\partial^* E_n\cap Q^N) <\frac{\rho}{3mM^2}\, a^{N-1}\qquad \forall\, n\notin \{0,\,i,\,j\} \,.
\end{equation}

\step{II}{Reduction to the validity of~(\ref{reallyassume})}
The goal of this step is to show that~(\ref{wecanassume}) can be extended also to the case $n=0$, if $0\notin \{i,\,j\}$. More precisely, we reduce to the case
\begin{equation}\label{reallyassume}
\H^{N-1}(\partial^* E_n\cap Q^N) <\rho a^{N-1}\qquad \forall\, n\notin \{i,\,j\} \,.
\end{equation}
Notice that an even stronger inequality is already given by~(\ref{wecanassume}) if $n\neq 0$, and there is nothing to prove if $0\in \{i,\,j\}$. We have then only to get the validity of~(\ref{reallyassume}) assuming that $n=0$ and that $i,\, j\neq 0$. By~(\ref{0ornot0}) and using~(\ref{wecanassume}) we can then write
\begin{equation}\label{stp2}\begin{split}
\H^{N-1}(\partial^* E_0\cap Q^N) &= \sum_{k\neq 0,\, i} \H^{N-1}(\partial^* E_0\cap \partial^* E_k \cap Q^N)\\
&\leq \frac{\rho}{3M^2}\, a^{N-1} + \H^{N-1}(\partial^* E_0\cap \partial^* E_j \cap Q^N)\,.
\end{split}\end{equation}
We call now $\Omega=E_0\cap Q^N$, so the above estimate and~(\ref{eq:5}) give
\[\begin{split}
\H^{N-1}(\partial^*\Omega) &\leq \H^{N-1}(E_0\cap \partial Q^N)+\H^{N-1}(\partial^* E_0 \cap Q^N)\\
&\leq \bigg( C_1 \rho^2 + \frac 1{3M^2}\bigg)\rho a^{N-1}+ \H^{N-1}(\partial^* E_0\cap \partial^* E_j \cap Q^N)\,.
\end{split}\]
Next, we define the modified cluster $\E'$ by putting $E_j'=E_j\cup \Omega$, $E_0'=E_0\setminus \Omega$, and $E_k'=E_k$ for every $k\notin \{0,\, j\}$. We have then
\[\begin{split}
P(\E)-P(\E') &\geq \frac {\H^{N-1}(\partial^* E_0\cap \partial^* E_j \cap Q^N)}M - M \Big( \H^{N-1}(\partial^*\Omega)-\H^{N-1}(\partial^* E_0\cap \partial^* E_j \cap Q^N)\Big)\\
&\geq \frac{\H^{N-1}(\partial^* E_0\cap \partial^* E_j \cap Q^N)}M\, - M\bigg( C_1 \rho^2 + \frac 1{3M^2}\bigg)\rho a^{N-1}\,.
\end{split}\]
Arguing as in the previous step, the proof is easily concluded if $P(\E)-P(\E')\geq \rho a^{N-1}/4M$, hence we reduce ourselves to consider the opposite case. That is, we can assume that
\[
\H^{N-1}(\partial^* E_0\cap \partial^* E_j \cap Q^N) \leq M^2\bigg( C_1 \rho^2 + \frac 1{3M^2}\bigg)\rho a^{N-1} + \frac 14\,\rho a^{N-1}\,.
\]
And by~(\ref{stp2}), this concludes the validity of~(\ref{reallyassume}) with $n=0$ as soon as $\rho$ is small enough.

\step{III}{The ``good'' part $G$ and the validity of~(\ref{GtuttoQ}) and~(\ref{perGtuttoQ}).}
In the first steps, we have selected a small cube $Q^N$ for which properties (\ref{eq:1})--(\ref{eq:5}) and~(\ref{reallyassume}) hold. Notice that, basically, these properties are quantifying through a small parameter $\rho$ the fact that the cluster $\E$ in the cube $Q^N$ is close to be given by $E_i$ in the lower half of the cube and $E_j$ in the upper half, and correspondingly $\partial^*\E$ in $Q^N$ is close to be given by the $(N-1)$-dimensional cube $Q$, as common boundary of $E_i$ and $E_j$. As a consequence we might expect that, for a generic $x'\in Q$, the vertical segment $\{x'\}\times (-a/2,a/2)$ should be contained in $E_i$ for the lower half and in $E_j$ for the upper half, having no intersection with the other sets $E_n$. We define then $G\subseteq Q$ the set of the points $x'\in Q$ for which this is more or less true. More precisely, for every $x'\in Q$ and every $0\leq n \leq m$ we set for brevity $E_{n,x'}=(E_n)_{x'} \cap (-a/2,a/2)$ and $\partial^* E_{n,x'}=\partial^*\big((E_n)_{x'}\big)\cap (-a/2,a/2)$. The ``good'' set $G$ is then defined as the set of those $x'\in Q$ such that
\begin{enumerate}[i)]
\item $\partial^*\big( (E_n)_{x'}\big) = \big(\partial^* (E_n)\big)_{x'}$ for every $0\leq n\leq m$;
\item $\partial^* E_{n,x'}=\emptyset$ for every $n\notin \{i,\,j\}$;
\item $\# \Big(\partial^* E_{i,x'}\Big)=\# \Big(\partial^* E_{j,x'}\Big)=1$;
\item $(-a/2,-a\rho)\subseteq E_{i,x'}\subseteq (-a/2,a\rho)$.
\end{enumerate}
We can now prove that $G$ contains most of the interesting information, that is, $Q\setminus G$ is a small portion of $G$, and its sections carry only a small portion of $\partial^*\E\cap Q^N$. Precisely, we claim that
\begin{gather}
\H^{N-1}(Q \setminus G) \leq (m+4)\rho a^{N-1}\,, \label{GtuttoQ}\\
\H^{N-1}\Big(\partial^*\E \cap \big( (Q\setminus G )\times (-a/2,a/2) \big)\Big)\leq (3m+11)\rho a^{N-1}\label{perGtuttoQ}\,.
\end{gather}
To obtain these estimates, let us call $\Gamma_1$ the set of points $x'\in Q$ such that the above property~i) fails, and similarly we call $\Gamma_2,\, \Gamma_3$ and $\Gamma_4$ the set of points such that the properties~ii), iii) and~iv) respectively fail. Vol'pert Theorem~\ref{volpert} ensures that
\begin{equation}\label{Gam1}
\H^{N-1}(\Gamma_1)=0\,,
\end{equation}
while~(\ref{reallyassume}) gives that
\begin{equation}\label{Gam2}
\H^{N-1}(\Gamma_2) < (m-1) \rho a^{N-1}\,,
\end{equation}
since for every $0\leq n \leq m$ we have
\[
\H^{N-1}(\partial^* E_n \cap Q^N) \geq \H^{N-1}\Big(\big\{ x'\in G:\, \partial^* E_{n,x'} \neq \emptyset\big\}\Big)\,.
\]
Let us now consider $\Gamma_3$, and let us write $\Gamma_3=\Gamma_3^{i,0}\cup \Gamma_3^{i,2}\cup \Gamma_3^{j,0}\cup \Gamma_3^{j,2}$, where $x'$ belongs to $\Gamma_3^{i,0}$ if $\# \Big(\partial^* E_{i,x'}\Big)=0$, and to $\Gamma_3^{i,2}$ if $\# \Big(\partial^* E_{i,x'}\Big)\geq 2$, and $\Gamma_3^{j,0}$ and $\Gamma_3^{j,2}$ are defined in the same way. If $x'\in \Gamma_3^{i,0}$, then either the whole section $\{x'\}\times (-a/2,a/2)$ is contained in $E_i$, or it does not intersect $E_i$, thus
\[
\H^1 \bigg(\Big( \big(\{x'\} \times \{x_N<0\}\big) \setminus E_i\Big) \cup \Big(\big(\{x'\} \times \{x_N>0\}\big) \setminus E_j \Big) \bigg)\geq a/2\,,
\]
and the same estimate is true if $x'\in \Gamma_3^{j,0}$. Then, by~(\ref{eq:3}),
\begin{equation}\label{Gam3part}
\H^{N-1}\big(\Gamma_3^{i,0}\cup\Gamma_3^{j,0} \big) < 2\rho^3 \,a^{N-1}\,.
\end{equation}
Observe now that
\[\begin{split}
\H^{N-1}\big(\partial^* E_i \cap Q^N\big) &\geq 2\H^{N-1}( \Gamma_3^{i,2}) + \H^{N-1}\big(Q \setminus (\Gamma_3^{i,0}\cup \Gamma_3^{i,2})\big)\\
&\geq a^{N-1} - \H^{N-1}(\Gamma_3^{i,0})+\H^{N-1}(\Gamma_3^{i,2})
\geq (1-2\rho^3) a^{N-1}+\H^{N-1}(\Gamma_3^{i,2})\,.
\end{split}\]
On the other hand, by~(\ref{eq:1}) and~(\ref{eq:2}) we have that
\begin{equation}\label{onlyonce}
\H^{N-1}\big(\partial^* E_i \cap Q^N\big) \leq (1+2\rho) a^{N-1}\,,
\end{equation}
which inserted in the above estimate gives
\[
\H^{N-1}(\Gamma_3^{i,2}) \leq 2(\rho+\rho^3) a^{N-1}\,.
\]
Since the same estimate holds for $\Gamma_3^{j,2}$, by~(\ref{Gam3part}) we deduce
\begin{equation}\label{Gam3}
\H^{N-1}(\Gamma_3) < (4 \rho+ 6\rho^3) a^{N-1}\,.
\end{equation}
Let us finally consider a point $x'\in \Gamma_4\setminus (\Gamma_1\cup\Gamma_2\cup\Gamma_3)$. The fact that $x'\notin (\Gamma_1\cup\Gamma_2\cup\Gamma_3)$ ensures that the segment $\{x'\}\times (-a/2,a/2)$ is the union of two segments, one contained in $E_i$ and the other contained in $E_j$. Then, since $x'\in\Gamma_4$, we have
\[
\H^1 \bigg(\Big( \big(\{x'\} \times \{x_N<0\}\big) \setminus E_i\Big) \cup \Big(\big(\{x'\} \times \{x_N>0\}\big) \setminus E_j \Big) \bigg)\geq \rho a\,,
\]
so again by~(\ref{eq:3}) we deduce
\[
\H^{N-1}\big(\Gamma_4\setminus (\Gamma_1\cup\Gamma_2\cup\Gamma_3) \big) < \rho^2 \,a^{N-1}\,.
\]
Putting this estimate together with~(\ref{Gam1}), (\ref{Gam2}) and~(\ref{Gam3}), and up to possibly decrease $\rho$, we deduce the validity of~(\ref{GtuttoQ}).\par

The estimate~(\ref{perGtuttoQ}) is then an easy consequence. Indeed, by~(\ref{onlyonce}) and~(\ref{GtuttoQ}), we have
\[\begin{split}
(1+2\rho) a^{N-1} &\geq \H^{N-1}\big(\partial^* E_i \cap Q^N\big) \\
&\geq \H^{N-1}\big(\partial^* E_i \cap (G\times (-a/2,a/2))\big) +\H^{N-1}\Big(\partial^* E_i \cap \big((Q\setminus G)\times (-a/2,a/2)\big)\Big) \\
&\geq \H^{N-1}(G)+\H^{N-1}\Big(\partial^* E_i \cap \big((Q\setminus G)\times (-a/2,a/2)\big)\Big) \\
&\geq \big(1- (m+4)\rho\big)a^{N-1}+\H^{N-1}\Big(\partial^* E_i \cap \big((Q\setminus G)\times (-a/2,a/2)\big)\Big) \,,
\end{split}\]
which implies
\[
\H^{N-1}\Big(\partial^* E_i \cap \big((Q\setminus G)\times (-a/2,a/2)\big)\Big) \leq (m+6)\rho a^{N-1}\,.
\]
Since the same estimate clearly holds with $j$ in place of $i$, keeping in mind also~(\ref{reallyassume}) we obtain~(\ref{perGtuttoQ}).
\step{IV}{Selection of the small cubes $Q_h\subseteq Q$.}
In this step we select some ``good'' cubes $Q_h$ in $Q$. We start by letting $L=L(M,N,m,K)$ be a large constant, to be specified later, and by calling for brevity $\ell = L \eps^{1/N}$, where $0<\eps<\bar\eps$ is any constant. Keep in mind that $L$ depends only on $M,\,N,\,m$ and $K$, as well as $a$ (in fact, $a$ depends also on $\rho$, but $\rho=\rho(M,\,N,\,m,\,K)$). As a consequence, up to select $\bar\eps\ll 1$, only depending on $M,\,N,\,m$ and $K$, we can assume that $\ell \ll a$, and we will use this assumption several times later.

Keep in mind that our goal is to find a cluster $\F$ satisfying~(\ref{oldone}). Now, for every $c\in Q$, we denote by $Q_\ell(c)$ (resp., $Q_{2\ell}(c)$) the $(N-1)$-dimensional cube in $\R^{N-1}$ with center $c$, side $\ell$ (resp., $2\ell$), and with sides parallel to the coordinate planes. We can now observe that, for any set $\Theta\subseteq\R^N$,
\begin{equation}\label{fattobene}\begin{split}
\int_{Q_\ell(c)} \H^{N-2} &\Big(\Theta \cap \big(\partial Q_\ell(x')\times (-a/2,a/2)\big)\Big) \, d\H^{N-1}(x')\\
&\leq (N-1)\ell^{N-2} \H^{N-1}\Big(\Theta \cap\big( Q_{2\ell}(c)\times (-a/2,a/2)\big)\Big)\,.
\end{split}\end{equation}
To show this inequality, for every direction $1\leq n\leq N-1$ and every $x'\in Q_\ell(c)$ we define
\[
S_n^\pm(x') = \big\{ y \in Q_{2\ell}(c):\, y_n = x_n' \pm \ell/2\big\}\,.
\]
Then, we observe that
\[\begin{split}
\int_{Q_\ell(c)} \H^{N-2}&\Big(\Theta \cap \big(S_n^-(x')\times (-a/2,a/2)\big)\Big)\,d\H^{N-1}(x')\\
&=\ell^{N-2} \int_{-\ell}^{0} \H^{N-2}\Big(\Theta \cap \big(Q_{2\ell}(c)\times (-a/2,a/2)\big)\cap \{x_n = c_n+t\}\Big) \, dt\\
&\leq \ell^{N-2}\H^{N-1}\Big(\Theta \cap \big(Q_{2\ell}(c)\times (-a/2,a/2)\big)\cap \Big\{c_n-\ell<x_n<c_n \,\Big\}\Big)\,.
\end{split}\]
Since the same estimate is clearly valid if we substitute $S_n^+$ to $S_n^-$, and correspondingly the interval $(c_n,c_n + \ell)$ to the interval $(c_n - \ell ,c_n)$, and since for every $x'\in Q_\ell(c)$ the boundary $\partial Q_\ell(x')$ is contained in the union of the sets $S_n^\pm$ with all $1\leq n\leq N-1$, adding the above estimates we get~(\ref{fattobene}).\par

Let us now take $2H$ points $\{\tilde x_1',\, \tilde x_2',\, \dots\,,\, \tilde x_{2H}'\}$ in $Q$ so that the open cubes $Q_{2\ell}(\tilde x_h')$ are disjoint and compactly contained in $Q$. Since $\ell\ll a$, this is clearly possible with some
\begin{equation}\label{estiH}
H\geq \frac{a^{N-1}}{2^{N+1}\ell^{N-1}}\,.
\end{equation}
We now recall that, by~(\ref{eq:1}), (\ref{eq:1bis}), (\ref{eq:2}) and~(\ref{reallyassume}), and keeping in mind that $\H^{N-1}$-a.e. point of $\partial^*\E$ belongs to exactly two different boundaries $\partial^* E_i$ with $0\leq i \leq m$, the estimate
\begin{equation}\label{easypar*E}
\H^{N-1}(\partial^*\E\cap Q^N) \leq \bigg(1 + \frac{\rho (m+2)}2\bigg) a^{N-1} \leq 2 a ^{N-1}
\end{equation}
holds. As a consequence, at most $H$ of the disjoint cubes $Q_{2\ell}(\tilde x_h')$ can satisfy
\[
\H^{N-1}\Big(\partial^*\E \cap\big( Q_{2\ell}(\tilde x_h')\times (-a/2,a/2)\big)\Big)\geq \frac{2 a^{N-1}}H\,.
\]
Up to renumbering, then, we can assume that all the $H$ cubes $Q_{2\ell}(\tilde x_h')$ with $1\leq h\leq H$ satisfy the opposite inequality, which by~(\ref{estiH}) becomes
\[
\H^{N-1}\Big(\partial^*\E \cap\big( Q_{2\ell}(\tilde x_h')\times (-a/2,a/2)\big)\Big)\leq \frac{2 a^{N-1}}H 
\leq 2^{N+2} \ell^{N-1}\,.
\]
Finally, applying~(\ref{fattobene}), for any $1\leq h\leq H$ we can select a point $x_h' \in Q_\ell(\tilde x_h')$ such that
\[\begin{split}
\H^{N-2} \Big(\partial^*\E &\cap \big(\partial Q_\ell(x_h')\times (-a/2,a/2)\big)\Big)\\
&\leq \frac 1{\ell^{N-1}} \int_{Q_\ell(\tilde x_h')} \H^{N-2} \Big(\partial^*\E \cap \big(\partial Q_\ell(x')\times (-a/2,a/2)\big)\Big) \, d\H^{N-1}(x')\\
&\leq \frac{N-1}\ell\, \H^{N-1}\Big(\partial^*\E \cap\big( Q_{2\ell}(\tilde x_h')\times (-a/2,a/2)\big)\Big)
\leq 2^{N+2} (N-1) \ell^{N-2}\,.
\end{split}\]
In other words, calling for brevity $Q_h = Q_\ell(x_h')$ we know that 
\begin{equation}\label{sectN-2small}
\H^{N-2}\Big(\partial^*\E \cap \big(\partial Q_h\times (-a/2,a/2)\big)\Big)\leq 2^{N+2} (N-1) \ell^{N-2} \qquad \forall\, 1\leq h\leq H\,.
\end{equation}

\step{V}{Selection of an horizontal cube $Q_\eps$.}
Our next aim is to select one of the cubes defined in the previous step. More precisely, we will write $Q_\eps=Q_j$ and $Q_\eps^N=Q_j\times (-a/2,a/2)$ for a suitable $1\leq j \leq H$, in such a way that
\begin{gather}
\frac{\H^{N-1}\big(\partial^*\E \cap Q^N_\eps\big)}{\ell^{N-1}} \leq 1+5\cdot 2^{N+1}(m+3) \rho\,,\label{eq:10new}\\
\frac{\H^{N-1}\Big(\Big(\partial^*\E \setminus \{-a\rho < x_N < a\rho\} \bigcup \cup_{h\notin \{i,\,j\}} \partial^* E_h \Big)\cap Q_\eps^N\Big)}{\ell^{N-1}} \leq 3\cdot 2^{N+1} m \rho\,,\label{eq:11new}\\
\frac{\H^{N-1}(Q_\eps\setminus G)}{\ell^{N-1}} \leq 3\cdot 2^{N+1}(m+4)\rho\,.\label{eq:13new}
\end{gather}
Let us start by considering the last estimate. There are at most $H/3$ of the cubes $Q_j$ for which
\[
\H^{N-1}(Q_j \setminus G) \geq \frac 3H\, \H^{N-1}(Q\setminus G) \,,
\]
hence using~(\ref{estiH}) and~(\ref{GtuttoQ}) we get that for at least $2/3$ of the $H$ cubes $Q_j$ the estimate~(\ref{eq:13new}) is valid with $Q_j$ in place of $Q_\eps$. The argument to obtain estimate~(\ref{eq:11new}) is the same, we only have to observe that, calling for brevity
\[
\Theta = \partial^*\E \setminus \{-a\rho < x_N < a\rho\} \bigcup \cup_{h\notin \{i,\,j\}} \partial^* E_h\,,
\]
by~(\ref{eq:2}) and~(\ref{reallyassume}) we have
\[
\H^{N-1}(\Theta\cap Q^N)\leq m \rho a^{N-1}\,,
\]
and then again for at least $2/3$ of the $H$ cubes $Q_j$ we must have
\[
\H^{N-1}\Big(\Theta\cap \big(Q_j\times(-a/2,a/2)\big)\Big) \leq \frac 3 H \, \H^{N-1}(\Theta\cap Q^N)\,,
\]
which by~(\ref{estiH}) reduces to the validity of~(\ref{eq:11new}) with $Q_j\times (-a/2,a/2)$ in place of $Q_\eps^N$.

Finally, to get~(\ref{eq:10new}) we call $A$ the projection over $Q$ of $\partial^*\E\cap Q^N$, and we notice that
\[
\H^{N-1}(\partial^*\E\cap Q^N) - \H^{N-1}(A) \geq \sum_{1\leq j\leq H} \H^{N-1}\Big(\partial^*\E\cap \big(Q_j\times(-a/2,a/2)\big)\Big) - \H^{N-1}(A\cap Q_j)\,.
\]
On the other hand, observing that by construction $A\supseteq G$, from~(\ref{easypar*E}) and~(\ref{GtuttoQ}) we get
\[
\H^{N-1}(\partial^*\E\cap Q^N)-\H^{N-1}(A) \leq \bigg(1 + \frac{\rho (m+2)}2\bigg) a^{N-1}- \H^{N-1}(G)
\leq \rho a^{N-1}\bigg( \frac{3m+10}2\bigg)\,,
\]
so we deduce that for at most $H/3$ of the cubes $Q_j$ one may have
\[
\H^{N-1}\Big(\partial^*\E\cap \big(Q_j\times(-a/2,a/2)\big)\Big) - \H^{N-1}(A\cap Q_j) \geq \frac{(9m+30)\rho a^{N-1}}{2H}\,.
\]
Once again, also by~(\ref{estiH}) we deduce that for at least $2/3$ of the $H$ cubes $Q_j$ one must have
\[\begin{split}
\H^{N-1}\Big(\partial^*\E\cap \big(Q_j\times(-a/2,a/2)\big)\Big) &\leq \H^{N-1}(A\cap Q_j) + 2^N (9m+30)\rho \ell^{N-1}\\
&\leq \big(1 + 2^N (9m+30)\rho\big) \ell^{N-1}\,,
\end{split}\]
which is stronger than~(\ref{eq:10new}) with $Q_j\times(-a/2,a/2)$ in place of $Q_\eps^N$. Putting everything together, we can finally select a suitable $1\leq j\leq H$ so that, calling $Q_\eps=Q_j$ and $Q_\eps^N=Q_j\times (-a/2,a/2)$, the estimates~(\ref{eq:10new}), (\ref{eq:11new}) and~(\ref{eq:13new}) hold. We conclude this step by noticing that, thanks to~(\ref{eq:10new}) and~(\ref{eq:13new}), we have
\begin{equation}\label{uselater}\begin{split}
\H^{N-1}\Big(\partial^*\E &\cap\big((Q_\eps\setminus G)\times (-a/2,a/2)\big)\Big)\\
&=\H^{N-1}\big(\partial^*\E \cap Q_\eps^N\big)-\H^{N-1}\Big(\partial^*\E \cap\big((Q_\eps\cap G)\times (-a/2,a/2)\big)\Big)\\
&\leq \H^{N-1}\big(\partial^*\E \cap Q_\eps^N\big)-\H^{N-1}\big(Q_\eps\cap G\big)
\leq 2^{N+4}(m+4) \rho \ell^{N-1}\,.
\end{split}\end{equation}

\step{VI}{Definition of the modified clusters $\F_\delta$.}
In this step we define the clusters $\F_\delta$, which are slight modifications of the cluster $\E$. One of them will then be selected as the searched cluster $\F$. To start, we define
\[
\bar\delta = \frac{2 M \eps}{\ell^{N-1}}
\]
and we set
\[
K:= 2\left \lfloor \frac a{6\bar \delta} \right \rfloor\,.
\]
Notice that $K\gg 1$ up to select $\bar\eps\ll 1$, depending on $M,\,N,\,m$ and $K$. We can then select $K$ constants $a\rho < \sigma_1 < \sigma_2 < \,\cdots\, < \sigma_K < \frac a2 - \bar\delta$ in such a way that the strips $\SS_k = Q_\eps \times (\sigma_k,\, \sigma_k + \bar\delta)$ are pairwise disjoint, being $Q_\eps$ the horizontal cube defined in Step~V. Thanks to~(\ref{eq:11new}), we have then
\[
\sum_{k=1}^K \H^{N-1}\big(\partial^*\E \cap \overline{\SS_k}\big)\leq \H^{N-1} \Big(\partial^*\E \cap \big(Q_\eps \times (a\rho,a/2)\big)\Big)\\
\leq 3\cdot 2^{N+1} m \rho \ell^{N-1}\,.
\]
We can then fix one of these strips, call it $\SS^+:= \SS_{\overline k}$, in such a way that
\begin{equation}\label{fixsp}
\H^{N-1}\big(\partial^*\E \cap \overline{\SS^+}\big) \leq
\frac{3\cdot 2^{N+1} m\rho\ell^{N-1}}K
\leq \frac{5\cdot 2^{N+2} m\rho\ell^{N-1}\bar\delta}a\,.
\end{equation}
We also call for brevity $\sigma^+ = \sigma_{\overline k}$.\par

We need now to define another constant $-a/2<\sigma^-<-a\rho$. To do so, we observe that by~(\ref{eq:11new}) and by Vol'pert Theorem
\[
3\cdot 2^{N+1} m \rho \ell^{N-1} \geq \H^{N-1}\Big(\partial^*\E\cap \big( Q_\eps^N\times (-a/2,-a\rho)\big)\Big)
\geq \int_{t=-a/2}^{-a\rho} \H^{N-2}\Big(\big(\partial^*\E\cap Q_\eps^N\big)^t\Big)\,dt\,.
\]
As a consequence, we can find some $\sigma^-\in (-a/2, -a\rho)$ such that Vol'pert Theorem holds for $(E_n)^{\sigma^-}$ for each $0\leq n\leq m$, and so that
\begin{equation}\label{nitt}
\H^{N-2}\Big(\big(\partial^*\E\cap Q_\eps^N\big)^{\sigma^-}\Big) \leq \frac{2^{N+4} m \rho \ell^{N-1}}a\,.
\end{equation}

We are now in position to define the modified cluster $\F_\delta$ for any given $0<\delta<\bar\delta$. For every $h\notin \{i,\,j\}$, we simply set $F_{h,\delta}=E_h$, and we also set $F_{i,\delta}=E_i$ and $F_{j,\delta}=E_j$ outside of the cylinder $\C=Q_\eps\times (\sigma^-,\sigma^++\delta)$. Within the cylinder, instead, we let
\begin{align}\label{defFijdelta}
F_{i,\delta} \cap \C = S \cap (E_i\cup E_j)\,, && F_{j,\delta} \cap\C = (\C\setminus S)\cap (E_i\cup E_j)\,,
\end{align}
where the set $S\subseteq \C$ is given by
\[
S=\big(Q_\eps \cap E_i^{\sigma^-} \big)\times (\sigma^-,\sigma^-+\delta) \cup \Big\{ (x',x_N+\delta):\, (x',x_N) \in Q_\eps \times (\sigma^-,\sigma^+) \cap E_i\Big\}\,.
\]
Notice that the set $F_{i,\delta}$ in the cylinder has been basically defined by ``stretching'' of height $\delta$ the section $E_i^{\sigma^-}\cap Q_\eps$, by ``translating'' of an height $\delta$ the part of $E_i$ in $Q_\eps\times (\sigma^-,\sigma^+)$, and then by ``squeezing'' the part of $E_i$ in $Q_\eps\times (\sigma^+,\sigma^++\delta)$.

\step{VII}{Volume estimate and selection of the competitor $\F$.}
In this step we show that there exists some $\bar\delta/(4M^2)<\delta<\bar\delta$ such that the cluster $\F=\F_\delta$ satisfies $|F_i|=|E_i|+\eps$ and $|F_j|=|E_j|-\eps$. Let us fix a point $x'\in G\cap Q_\eps$. Then, by definition the whole section $\{x'\}\times (-a/2,a/2)$ is contained in $E_i\cup E_j$, and more precisely $\{x'\}\times(-a/2,y) \subseteq E_i$ and $\{x'\}\times (y,a/2)\subseteq E_j$ for some $-a\rho<y<a\rho$. As a consequence,
\begin{equation}\label{ifinG}
\H^1\Big(\big(\{x'\}\times (-a/2,a/2)\big)\cap (S\Delta E_i)\Big)=\H^1\Big(\big(\{x'\}\times (-a/2,a/2)\big)\cap (S\cap E_j)\Big) = \delta\,.
\end{equation}
Let instead $x' \in Q_\eps \setminus G$, and let $\sigma^- <y<\sigma^++\delta$ be such that the segment $\{x'\}\times (y-\delta,y)$ does not intersect $\partial^*\E$. Then, the whole segment belongs to a same set $E_h$, and then either $(x',y)\in E_i\cap S$, or $(x',y)\notin E_i\cup S$. Therefore, by~(\ref{ifinG}) and~(\ref{uselater}), on one hand we have
\[\begin{split}
\big|F_{i,\delta}\big|-|E_i|&\leq\delta M\H^{N-1}(G\cap Q_\eps)+\delta M \H^{N-1}\Big(\partial^*\E \cap \big((Q_\eps\setminus G)\times (\sigma^-,\sigma^+)\big)\Big)\\
&\leq\delta M \ell^{N-1} \big(1 + 2^{N+4}(m+4) \rho\big)\,,
\end{split}\]
and on the other hand also by~(\ref{eq:13new}) we have
\[\begin{split}
\big|F_{i,\delta}\big|-|E_i|&\geq \frac\delta M\,\H^{N-1}(G\cap Q_\eps)-\delta M \H^{N-1}\Big(\partial^*\E \cap \big((Q_\eps\setminus G)\times (\sigma^-,\sigma^+)\big)\Big)\\
&\geq \frac \delta M\, \ell^{N-1} \big( 1 - 2^{N+5}(m+4)M^2\rho\big)\,.
\end{split}\]
As a consequence, as soon as $\rho\ll 1$ we have
\begin{align*}
\big|F_{i,\bar\delta}\big|-|E_i| >\eps\,, && \Big|F_{i,\frac{\bar\delta}{4M^2}}\Big|-|E_i| <\eps\,,
\end{align*}
and then we can by continuity there is some $\bar\delta/(4M^2)<\delta<\bar\delta$ such that $|F_{i,\delta}|=|E_i|+\eps$. From now on, we will call $\F=\F_\delta$ with such a choice of $\delta$. Notice that, by construction and since $a<r$, the cluster $\F$ satisfies the first two requirements of~(\ref{oldone}), thus to conclude the thesis we only have to check the last requirement.

\step{VIII}{Perimeter estimate.}
This last step is devoted to obtain a perimeter estimate for the cluster $\F$. First of all we notice that, since Vol'pert Theorem holds for each $(E_n)^{\sigma^-}$, so in particular $\H^{N-1}$-a.e. point of $Q_\eps\times\{\sigma^-\}$ has density either $0$ or $1$ for each of the sets $E_n$, then
\begin{equation}\label{bottom}
\H^{N-1}\Big(\partial^*\F\cap \big(Q_\eps\times \{\sigma^-\}\big)\Big)=\H^{N-1}\Big(\partial^*\E\cap \big(Q_\eps\times \{\sigma^-\}\big)\Big)=0\,.
\end{equation}
Second, we observe that the perimeter of $\F$ can be written as
\[
P(\F) = \int_{\partial^*\F} \varphi_\F(x)\, d\H^{N-1}(x)
\]
for a suitable function $\varphi_\F:\partial^*\F\to \R^+$. More precisely, $\H^{N-1}$-almost any $x\in\partial^*\F$ belongs to the boundary of exactly two of the sets $F_h$ with $0\leq h_1<h_2\leq m$, and the corresponding outer normals are $\nu=\nu_{F_{h_1}}(x)$ and $-\nu=\nu_{F_{h_2}}(x)$ for a vector $\nu=\nu(x)\in\S^{N-1}$. The value of $\varphi_\F(x)$ is then either $g(x,-\nu)$ or $(g(x,\nu)+g(x,-\nu))/2$, respectively if $h_1=0$ and if $h_1>0$. Of course, the very same can be said for the perimeter of $\E$, obtaining another function $\varphi_\E:\partial^*\E\to\R^+$. Now, keep in mind that $F_h=E_h$ for every $h\notin\{i,\,j\}$, and call
\[
\Gamma = \cup_{h\notin \{i,\,j\} } \partial^* F_h \cap \overline\C = \cup_{h\notin \{i,\,j\} } \partial^* E_h \cap \overline\C\,.
\]
Since $\F\equiv \E$ outside of the closed cylinder $\overline\C$, we have
\[
P(\F) - P(\E) = \int_{\partial^*\F\setminus \Gamma} \varphi_\F(x)- \int_{\partial^*\E\setminus \Gamma} \varphi_\E(x)+\int_\Gamma \varphi_\F(x)-\varphi_\E(x)\,.
\]
Now, notice that the last term does not necessarily cancel, since $\varphi_\F$ and $\varphi_E$ might not coincide on $\Gamma$. Indeed, a point $x\in\Gamma$ might belong to $\partial^* E_h\cap\partial^* E_i\cap \partial^* F_j$, or to $\partial^* E_h\cap\partial^* E_j\cap \partial^* F_i$, and this causes a different value of $\varphi_\F(x)$ and $\varphi_\E(x)$ if $i=0$. However, since both $\varphi_\E$ and $\varphi_\F$ are bounded between $1/M$ and $M$ in $\overline\C$, the above estimate together with~(\ref{eq:11new}) gives
\begin{equation}\label{onlyij}
P(\F)-P(\E) -\bigg( \int_{\partial^*\F\setminus \Gamma} \varphi_\F(x)- \int_{\partial^*\E\setminus \Gamma} \varphi_\E(x) \bigg)\leq M \H^{N-1}(\Gamma) 
\leq 3\cdot 2^{N+1} M m \rho\ell^{N-1}\,.
\end{equation}
We have then only to concentrate ourselves on points of $\overline\C$ belonging to $\partial^*\F\setminus\Gamma$, which are then necessarily contained in $\partial^* F_i\cap \partial^* F_j$. Let then $x=(x',y)\in \C\cap(\partial^*\F\setminus\Gamma)$, and let us first assume that $\sigma^- < y\leq \sigma^-+\delta$. By the definition~(\ref{defFijdelta}) of $F_i$ and $F_j$, and by the fact that $x\in \partial^* F_i\cap\partial^* F_j$, we deduce that necessarily $(x',\sigma^-)\in \partial^* E_i\cap\partial^* E_j$. Therefore, also by~(\ref{nitt})
\begin{equation}\label{midbottom}\begin{split}
\int_{(Q_\eps\times(\sigma^-,\sigma^-+\delta])\cap (\partial^*\F\setminus\Gamma)} &\varphi_\F(x) \leq M \H^{N-1}\Big((Q_\eps\times(\sigma^-,\sigma^-+\delta])\cap (\partial^*\F\setminus\Gamma)\Big)\\
&\hspace{-20pt}\leq M\delta \H^{N-2}\Big(\big(\partial^* E_i\cap\partial^* E_j\big)\cap \big(Q_\eps\times\{\sigma^-\}\big)\Big)
\leq \frac{2^{N+4} mM\delta \rho \ell^{N-1}}a\,.
\end{split}\end{equation}
Now, consider a point $x=(x',\sigma^++\delta)\in\C\cap(\partial^*\F\setminus\Gamma)$. By construction and by Vol'pert Theorem again, there must be some $\sigma^+\leq y\leq \sigma^++\delta$ such that $(x',y)\in\partial^*\E$. As a consequence, by~(\ref{fixsp})
\begin{equation}\label{top}\begin{split}
\int_{(Q_\eps\times \{\sigma^++\delta\})\cap (\partial^*\F\setminus\Gamma)} \varphi_\F(x) &\leq M \H^{N-1} \Big(\partial^*\E \cap \big(Q_\eps\times [\sigma^+,\sigma^++\delta]\big)\Big)\,,\\
&\leq \frac{5\cdot 2^{N+2} m M \rho\ell^{N-1}\bar\delta}a\,.
\end{split}\end{equation}
The next stage is to consider a point $x=(x',y)\in (\C\cap \partial^*\F\setminus\Gamma)\cap \big(Q_\eps\times(\sigma^-+\delta,\sigma^++\delta) \big)$. As noticed above, we have $x\in\partial^* F_i\cap \partial^* F_j$. Moreover, by construction we also have $(x',y-\delta)\in \partial^* E_i\cap\partial^* E_j$. We claim that the equality
\begin{equation}\label{samenormal}
\nu_{F_i}(x',y) = \nu_{E_i}(x',y-\delta)
\end{equation}
holds, up to neglect a $\H^{N-1}$-negligible subset of $(\C\cap \partial^*\F\setminus\Gamma)\cap \big(Q_\eps\times(\sigma^-+\delta,\sigma^++\delta) \big)$. Indeed, let us call for a moment $F_i^+ = S \supseteq F_i$. In the cylinder $Q_\eps\times(\sigma^-+\delta,\sigma^++\delta)$, the set $F_i^+$ is nothing else than a vertical translation by a quantity $\delta$ of the set $E_i$, and then since $(x',y-\delta)\in \partial^* E_i$ we deduce that $(x',y)\in \partial^* F_i^+$ with $\nu_{F_i^+}(x',y) = \nu_{E_i}(x',y-\delta)$. Since $F_i^+\supseteq F_i$ but $x\in \partial^* F_i\cap \partial^* F_i^+$, we have also that $\nu_{F_i}(x) = \nu_{F_i^+}(x)$, and then~(\ref{samenormal}) is proved. As a consequence, keeping in mind Definition~\ref{defomegaxt} and~(\ref{eq:10new}) we can evaluate
\begin{equation}\label{interior}\begin{split}
\int_{(\partial^*\F\setminus\Gamma)\cap (Q_\eps\times (\sigma^-+\delta,\sigma^++\delta))} &\varphi_\F(x)\\
&\hspace{-80pt}\leq \int_{(\partial^* E_i\cap\partial^* E_j)\cap (Q_\eps\times (\sigma^-,\sigma^+))} \varphi_\E(x) + \omega_{\bar x}\bigg(\frac{a\sqrt N}2\bigg) \H^{N-1}\big(\partial^*\E\cap \C\big)\\
&\hspace{-80pt}\leq \int_{(\partial^* E_i\cap\partial^* E_j)\cap (Q_\eps\times (\sigma^-,\sigma^+))} \varphi_\E(x) + \omega_{\bar x}\bigg(\frac{a\sqrt N}2\bigg) 
\big(1+5\cdot 2^{N+1}(m+3) \rho\big)\ell^{N-1}\,.
\end{split}\end{equation}
The last set of points that we have to consider are the points on $\partial Q_\eps\times (\sigma^-,\sigma^++\delta) \cap \partial^*\F\setminus\Gamma$. For any such $(x',y)$, there must be by construction some point $(x',z)\in\partial^*\E$ with $(y-\delta)\vee \sigma^-\leq z\leq y$. As a consequence, by~(\ref{sectN-2small}) we get
\begin{equation}\label{lateral}\begin{split}
\H^{N-1}\big((\partial^*\F\setminus\Gamma)\cap (\partial Q_\eps\times (\sigma^-,\sigma^++\delta))\big) &\leq \delta \H^{N-1}\big(\partial^*\E\cap (\partial Q_\eps\times (-a/2,a/2))\big)\\
&\leq 2^{N+2} (N-1) \ell^{N-2}\delta\,.
\end{split}\end{equation}
We can finally put together the estimates~(\ref{bottom}), (\ref{onlyij}), (\ref{midbottom}), (\ref{top}), (\ref{interior}) and~(\ref{lateral}), recalling that $\delta<\bar\delta=2 M \eps/\ell^{N-1}$, and writing for brevity $\bar\omega$ in place of $\omega_{\bar x}(a\sqrt N/2)$, readily obtaining
\[
P(\F)-P(\E)\leq 2^{N+3} M m \rho\ell^{N-1}+\frac{2^{N+7} m M^2 \rho\eps}a+\frac{2^{N+3} N M^2 \eps}\ell+\bar\omega\big(1+5\cdot 2^{N+1}(m+3) \rho\big)\ell^{N-1}\,.
\]
Now, since $\ell = L \eps^{1/N}$, this inequality can be rewritten as
\[
\frac{P(\F)-P(\E)}{\eps^{\frac{N-1}N}} \leq 2^{N+3}Mm L^{N-1}\rho + \frac{2^{N+7}mM^2\rho\eps^{1/N}}a + \frac{2^{N+3}NM^2}L + \bar\omega\big(1+5\cdot 2^{N+1}(m+3) \rho\big) L^{N-1}\,.
\]
To avoid distinguishing the cases whether or not $\omega_{\bar x}=0$, where $\omega_{\bar x}=\lim_{t\searrow 0} \omega_{\bar x}(t)$ is the constant given by Definition~\ref{defomegaxt}, we write for brevity
\begin{equation}\label{defho}
\hat\omega:= \left\{\begin{array}{cc}
\omega_{\bar x} &\hbox{if $\omega_{\bar x}>0$}\,,\\[5pt]
\bal\frac {K^N}{(2^{N+3}NM^2+6)^N}\eal &\hbox{if $\omega_{\bar x}=0$}\,,
\end{array}\right.
\end{equation}
so that in any case $\hat\omega>0$ and $\hat\omega\geq \omega_{\bar x}$. We need now to choose the constants $L$ and $\rho$, only depending on $M,\,N,\,m$ and $K$, as well as the constants $a$ and $\bar\eps$, depending also on $r$. The first constant to be chosen is $L$, even though its value will be specified in few lines. Once chosen $L$, we can take $\rho$ so small that
\begin{align}\label{bestrho}
2^{N+3} M m L^{N-1} \rho < \sqrt[N]{\hat\omega}\,, && 5\cdot 2^{N+1}(m+3) \rho <1\,.
\end{align}
With this choice of $\rho$, the above inequality gives now
\[
\frac{P(\F)-P(\E)}{\eps^{\frac{N-1}N}} \leq \sqrt[N]{\hat\omega} + \frac{2^{N+7}mM^2\rho\eps^{1/N}}a + \frac{2^{N+3}NM^2}L + 2 \bar\omega L^{N-1}\,.
\]
Having set $\rho$, we can select $a\ll 1$ as precised in Step~I, with the additional requirement that
\[
\bar\omega=\omega_{\bar x}\bigg(\frac{a\sqrt N}2\bigg) < 2\hat\omega\,,
\]
which is possible since $2\hat\omega > \omega_{\bar x}$. Keep in mind that $a$ depends also on $r$, since in Step~I we required in particular that $a<r$. Therefore, the inequality is now
\[
\frac{P(\F)-P(\E)}{\eps^{\frac{N-1}N}} \leq \sqrt[N]{\hat\omega} + \frac{2^{N+7}mM^2\rho\eps^{1/N}}a + \frac{2^{N+3}NM^2}L + 4\hat\omega L^{N-1}\,.
\]
Since $L,\,\rho$ and $a$ are fixed, we can now take $\bar\eps\ll 1$ in such a way that
\begin{equation}\label{qualeps}
\frac{2^{N+7}mM^2\rho\bar\eps^{1/N}}a < \sqrt[N]{\hat\omega}\,,
\end{equation}
so that the estimate becomes
\[
\frac{P(\F)-P(\E)}{\eps^{\frac{N-1}N}} \leq 2\sqrt[N]{\hat\omega} + \frac{2^{N+3}NM^2}L + 4\hat\omega L^{N-1}\,.
\]
Of course, $\bar\eps=\bar\eps(M,\,N,\,m,\,K,\,r)$. Finally, a quick look at the inequality ensures that the best choice for $L$, up to a multiplicative constant, is given by
\begin{equation}\label{bestL}
L=\hat\omega^{-\frac 1N}\,.
\end{equation}
This gives to our estimate the final form
\[
\frac{P(\F)-P(\E)}{\eps^{\frac{N-1}N}} \leq \big(2^{N+3}NM^2+6\big) \sqrt[N]{\hat\omega}\,.
\]
We can finally define the constant $C=C(M_{\bar x},N,m)$ of the claim as $C=2^{N+3}NM^2+6$. In this way, the last estimate in~(\ref{oldone}) has now been established. Indeed, if $\omega_{\bar x}>0$ then $\hat\omega=\omega_{\bar x}$, and
\[
P(\F)-P(\E) \leq C \sqrt[N]{\omega_{\bar x}} \eps^{\frac{N-1}N} \leq K \eps^{\frac{N-1}N}\,,
\]
since $K> C \sqrt[N]{\omega_{\bar x}}$ by assumption. Observe that the value of $K$ does not actually play any role in this first case, since one could also simply take $K=C\sqrt[N]{\omega_{\bar x}}$. Instead, if $\omega_{\bar x}=0$, then by the definition~(\ref{defho}) of $\hat\omega$ we have
\[
P(\F)-P(\E) \leq C \sqrt[N]{\hat\omega} \eps^{\frac{N-1}N} = K \eps^{\frac{N-1}N}\,.
\]
Observe that, in this second case (which is often the most useful), the choice of $K$ has an effect, in particular if $K$ becomes extremely small then the same must happen to $\bar\eps$.\par

The proof is then concluded.
\end{proof}

\begin{remark}\label{eps<0}
Notice that, in the claim of Lemma~\ref{lemmaa=0}, we have considered $\eps>0$ just for simplicity of notations. But of course, since the role of $i$ and $j$ and be exchanged, the claim actually holds for every $-\bar\eps<\eps<\bar\eps$. With a non necessarily positive $\eps$, the third property of~(\ref{oldone}) has clearly to be written as $P(\F)\leq P(\E)+ K |\eps|^{\frac{N-1}N}$.
\end{remark}

Having proved Lemma~\ref{lemmaa=0}, it is rather easy to modify a cluster in such a way that only one of the volumes $|E_n|$ changes. More precisely, we have the following result.
\begin{lemma}\label{solouno}
Let $\E$ be an $m$-cluster, and let $A\subseteq\R^N$ be a bounded, open set so that for any $0\leq i<j\leq m$ one has that $\H^{N-1}(\partial^* E_i\cap \partial^* E_j)>0 \Longleftrightarrow \H^{N-1}(\partial^* E_i\cap \partial^* E_j\cap A)>0$. Then there exist an arbitrarily small constant $\overline r>0$, a positive constant $\bar \eps>0$, and $m(m+1)/2$ disjoint balls of radius at most $\overline r$ such that, for every $1\leq h \leq m$ and every $-\bar\eps<\eps<\bar\eps$, there is another cluster $\F$, which equals $\E$ outside of the union of the balls, and such that for a suitable constant $K$
\begin{align}\label{ilanat}
|F_n| = \left\{
\begin{array}{cc}
|E_n|+\eps &\hbox{if $n=h$}\,,\\
|E_n| &\hbox{if $n\neq h$}\,,
\end{array}\right.
&& P(\F) \leq P(\E) + K |\eps|^{\frac{N-1}N}\,.
\end{align}
More precisely, one can take any constant $K$ strictly bigger than $C \sqrt[N]{\omega_A}$, where $\omega_A=\sup_{x\in A} \omega_x$ and the constant $C$ only depends on $N$, on $m$ and on $M=\sup_{x\in A} M_x$.
\end{lemma}
\begin{proof}
In the set of indices $\{0,\, 1,\, \dots\, ,\, m\}$, we write that $i\sim j$ whenever $\H^{N-1}(\partial^* E_i\cap \partial^* E_j)>0$ and assumption~(\ref{0ornot0}) holds, up to swap $i$ and $j$, and we consider the weakest equivalence relation $\approx$ such that $i\sim j\Rightarrow i\approx j$. We claim that $i\approx j$ for every pair $i,\, j$. Indeed, let $\C$ be a non-empty equivalence class, and let $G=\cup_{i\in\C} E_i$. Up to $\H^{N-1}$-negligible subsets, points of $\partial^* G$ are the points which belong to exactly one of the boundaries $\partial^* E_i,\, i\in\C$. As a consequence, if $\partial^* G$ is not $\H^{N-1}$-negligible, then there must be some $i\in\C$ and $j\notin \C$ such that $\H^{N-1}(\partial^* E_i\cap \partial^* E_j)>0$. The fact that $i\not\approx j$, so in particular $i\not\sim j$, implies that assumption~(\ref{0ornot0}) does not hold for $i$ and $j$, not even swapping the role of the two indices. This means that $i\neq 0$ and $j\neq 0$, and moreover $\H^{N-1}(\partial^* E_i\cap \partial^* E_0)>0$ and $\H^{N-1}(\partial^* E_j\cap \partial^* E_0)>0$. In turn, this implies that $i\sim 0$ and $j\sim 0$, thus $i\approx j$ and this is a contradiction. Summarizing, we have proved that $\H^{N-1}(\partial^* G)=0$, and since $G$ is not empty we deduce that $G=\R^N$, that is, all the indices are equivalent.\par

For every two indices $i,\,j$ such that $i\sim j$, we select now a point $\bar x_{i,j}\in A\cap \partial^* E_i\cap\partial^* E_j$. We let then $\bar r$ be a length much smaller than the distance between any two of these points, and we aim to apply Lemma~\ref{lemmaa=0} to each of these points. Since the quantities $M_{\bar x_{i,j}}$ are bounded by $M$, we find a constant $C_0=C_0(M,N,m)$ which is bigger than each $C(M_{\bar x_{i,j}},N,m)$ from Lemma~\ref{lemmaa=0}. We let than $K_0$ be a constant such that, for any of the points $\bar x_{i,j}$, one has
\[
K_0 > C_0 \sqrt[N]{\omega_A}
\geq C(M_{\bar x_{i,j}},N,m) \sqrt[N]{\omega_{\bar x_{i,j}}}\,.
\]
It is then possible to apply Lemma~\ref{lemmaa=0} with constants $K_0$ and $\bar r$ to all the points $\bar x_{i,j}$, and we define $\bar\eps$ to be the minimum of the resulting constants $\bar\eps_{i,j}$. Let us now fix an index $1\leq h\leq m$ and a constant $\bar\eps<\eps<\bar\eps$. We have that $h\approx 0$, and then there is a sequence of indices $h_1 = h,\, h_2,\, h_3,\, \dots\, ,\, h_P=0$, all distinct, such that $h_l\sim h_{l+1}$ for every $1\leq l <P$. By Lemma~\ref{lemmaa=0}, and also keeping in mind Remark~\ref{eps<0} if $\eps<0$, for every $1\leq l < P$ there is a cluster $\F^l$ which equals $\E$ outside of the ball $B_l$ with radius $\bar r$ centered at $\bar x_{h_l,h_{l+1}}$, such that
\begin{align*}
|F^l_n| = \left\{
\begin{array}{cc}
|E_n|+\eps &\hbox{if $n=h_l$}\,,\\
|E_n|-\eps &\hbox{if $n=h_{l+1}$}\,,\\
|E_n| &\hbox{if $n\notin \{h_l,h_{l+1}\}$}\,,
\end{array}\right.
&& P(\F_l) \leq P(\E) + K_0 |\eps|^{\frac{N-1}N}\,.
\end{align*}
Since the balls $B_l$ are all disjoint, we define then $\F$ to be the cluster which equals $\F^l$ on each ball $B_l$, and $\E$ outside of the union of the balls. Since by construction we have $P\leq m+1$, we immediately deduce the validity of~(\ref{ilanat}) with $K=mK_0$. The thesis is then concluded with $C=mC_0$.
\end{proof}

As an immediate corollary, we obtain then the validity of Theorem~\ref{main} in the case $\alpha=0$; in fact, one needs to use at most $m(m+1)/2$ balls, while $m(m+1)/2+1$ would be acceptable by Definition~\ref{defpropeb}. The precise estimate of how the constant $\Cper[t]$ depends on $g$ is made in Remark~\ref{bestC}.

\section{Second step: the Infiltration Lemma\label{sec:infil}}

As already said in the introduction, our next step is to prove an ``Infiltration Lemma''. What we really need is to show that, if a point $x$ belongs to the reduced boundary of two different sets $E_i$ and $E_j$, then we can reduce ourselves to the case when a sufficiently small ball around $x$ only intersects $E_i$ and $E_j$ (and this is actually already the case if the cluster is minimal and the $\eps-\eps^{\frac{N-1}N}$ property with an arbitrarily small constant holds). The main difficulty to prove such a result is due to the fact that $g$ is not assumed to be symmetric in the angular variable, that is, $g(y,\nu)$ and $g(y,-\nu)$ need not to coincide. If one makes this additional assumption, then the proof can be obtained more or less with the same argument as in the book~\cite{Maggibook}, which considers the Euclidean case when $f\equiv g\equiv 1$, only with few technical complications.\par

We start with a simple observation.
\begin{lemma}\label{lemmaint}
Let $G\subseteq \R^N$ be a set of finite perimeter, let $x$ be a point of density $0$ for $G$, and let $H$ be a large constant. Then, there exists $r=r(x,G,H)\ll 1$ such that
\begin{equation}\label{rr2}
\H^{N-1}\Big(\partial B(x,r)\cap\big(G\cup \partial^* G\big)\Big) \leq \frac 1H \, \H^{N-1}\big(\partial^* (G\cap B(x,r))\big)\,.
\end{equation}
\end{lemma}
\begin{proof}
Let us call $m(t)=\H^N(G \cap B(x,t))$, and notice that the claim is emptily true if there is some $t>0$ such that $m(t)=0$. Therefore, we can assume that $m(t)>0$ for every $t>0$. Since $x$ is a point of density $0$ for $G$, there exist $\overline r\ll 1$ such that
\begin{equation}\label{choiceovr}
\H^N (G \cap B(x,\overline r))\leq \frac{\omega_N}{(4H)^N}\, \overline r^N\,.
\end{equation}
We claim that there exist some $\overline r/2 \leq r \leq \overline r$ which satisfies~(\ref{rr2}). Indeed, if our claim is false then for a.e.\ $\overline r/2 \leq t \leq \overline r$ we have
\[\begin{split}
m'(t) &= \H^{N-1}\Big(\partial B(x,t)\cap\big(G\cup \partial^* G\big)\Big) > \frac 1H \, \H^{N-1}\big(\partial^* (G\cap B(x,t)\big)\\
&\geq \frac{N\omega_N^{1/N}}H\, \H^N(G\cap B(x,t))^{\frac{N-1}N}
= \frac{N\omega_N^{1/N}}H\, m(t)^{\frac{N-1}N}\,,
\end{split}\]
which for every $\overline r/2 < t <\overline r$, since $m(t)>0$, can be rewritten as
\[
\Big(m(t)^{1/N}\Big)' > \frac{\omega_N^{1/N}}H\,. 
\]
Integrating this inequality between $\overline r/2$ and $\overline r$ and keeping in mind~(\ref{choiceovr}) we obtain that
\[
\frac{\omega_N^{1/N}}H \, \cdot\, \frac{\overline r}2 \leq m(\overline r)^{1/N} - m(\overline r/2)^{1/N}\leq \frac{\omega_N^{1/N}}{4H}\, \overline r\,,
\]
which is impossible. The proof is then concluded.
\end{proof}

\begin{remark}
Notice that, calling $m(t) = \H^N(G \cap B(x,t))$ as above, for a.e. $t$ we have
\[
m'(t) = \H^{N-1}\Big(\partial B(x,t)\cap\big(G\cup \partial^* G\big)\Big) = \H^N\big(G\cap \partial B(x,t)\big)\,,
\]
hence we could have written $\H^{N-1}(G\cap\partial B(x,r))$ in the l.h.s. of~(\ref{rr2}). We preferred to write the estimate in its less nice form just because it is more convenient in the proof of Lemma~\ref{InfilLem}.
\end{remark}

Let us now present the result that we need.

\begin{lemma}[Infiltration Lemma]\label{InfilLem}
Let $\E$ be an $m$-cluster, let $i,\, j \in \{0,\, 1,\, \dots\,,\, m\}$ be two indices, let $x\in \partial^* E_i\cap\partial^* E_j$, and let us assume that, up to possibly swap $i$ and $j$,
\begin{align}\tag{\ref{0ornot0}}
\hbox{\rm either} \quad i=0,\, && \hbox{\rm or} \quad \H^{N-1}\big(\partial^* E_i \cap \partial^* E_0\big)=0\,.
\end{align}
Then, there exist $r \ll 1$, $C=C(x,N,\E)$ and a cluster $\F$ so that
\begin{align}\label{thilvtrue}
\big|B(x,r) \setminus \big(F_i\cup F_j\big)\big|=0\,, &&\F \Delta \E \comp B(x,r) \,, && P(\F) \leq P(\E) - C \big| \F\Delta\E\big|^{\frac{N-1}N}\,.
\end{align}
\end{lemma}

\begin{remark}
The Infiltration Lemma is particularly interesting if $\E$ is a minimal cluster and the $\eps-\eps^{\frac{N-1}N}$ Lemma holds with a constant strictly smaller than $C$, which is the case if $g$ is continuous in the first variable (in fact, in this case $\omega_x=0$ and then the $\eps-\eps^{\frac{N-1}N}$ Lemma holds with any positive constant by Lemma~\ref{lemmaa=0}). Indeed, the validity of~(\ref{thilvtrue}) and the $\eps-\eps^{\frac{N-1}N}$ Lemma with constant smaller than $C$ immediately imply that the cluster $\F$ given by the Infiltration Lemma actually coincides with $\E$ for $r$ small enough. Therefore, we deduce that in a sufficiently small ball the minimal cluster $\E$ only intersects $E_i$ and $E_j$.
\end{remark}

\proofof{Lemma~\ref{InfilLem}}
Let us define a constant $M=M(x,\E)\geq 1$ so that $f(y), \, g(y,\nu) \in [1/M, M]$ for every $y\in B(x,1)$ and every $\nu\in\S^{N-1}$. Then, call $G=\R^N\setminus(E_i\cup E_j)$, and apply Lemma~\ref{lemmaint} to the set $G$ with constant $H=10M^4$, which is possible since by construction $x$ is a point of density $0$ for $G$. This provides us with a small $r\ll 1$ so that~(\ref{rr2}) holds. We call ``infiltration'' the set $I=B(x,r)\setminus \big(E_i\cup E_j\big)$, and we subdivide $\partial^* I = D \cup \Gamma_1 \cup \Gamma_2$, where the three essentially disjoint sets $D,\,\Gamma_1$ and $\Gamma_2$ are defined as
\begin{align*}
D = \partial^* I \cap \partial B(x,r)\,, && \Gamma_1=\big(\partial^* I \cap \partial^* E_i\big)\setminus D\,, && \Gamma_2 = \big(\partial^* I \cap \partial^* E_j\big)\setminus D\,.
\end{align*}
Observe that $I=G\cap B(x,r)$ and $\partial^* I\cap \partial B(x,r) \subseteq \partial^* G\cup G$ up to $\H^{N-1}$-negligible subsets, hence~(\ref{rr2}) implies
\[
\H^{N-1}(D)\leq \frac 1{10M^4} \, \H^{N-1}\big(\partial^* I\big)\,.
\]
As a consequence, since $M\geq 1$ one readily finds that either
\begin{equation}\label{firstcasenew}
\H^{N-1}(\Gamma_1) > 2M^2 \H^{N-1}(D\cup \Gamma_2)
\end{equation}
or
\begin{equation}\label{secondcasenew}
\H^{N-1}(\Gamma_2) > 2M^2 \H^{N-1}(D)\,.
\end{equation}
We argue separately in the two cases. First of all, we assume that~(\ref{firstcasenew}) holds, and we define the cluster $\F$ by putting $F_i=E_i\cup I$, and $F_h=E_h\setminus I$ for every $h\in \{0,\, 1,\, \dots\, ,\, m\}\setminus \{i\}$. The first two requirements of~(\ref{thilvtrue}) are clearly satisfied. Concerning the last one, we first notice that in $\R^N\setminus\partial^* I$ the perimeters of $\E$ and $\F$ coincide by construction. In addition, the set $\Gamma_1$ is contained in $\partial^*\E\setminus \partial^*\F$, while $\partial^*\F\setminus\partial^*\E$ is contained in $D$. Moreover $\Gamma_2$ is contained both in $\partial^*\E$ and $\partial^*\F$, but it may have different weights, since $g$ is not assumed to be symmetric, and it might be that a same point of $\Gamma_2$ belongs $\partial^* E_0\Delta \partial^* F_0$, so its contribution to $P(\E)$ and $P(\F)$ might be different. However, by~(\ref{firstcasenew}) we can estimate
\begin{equation}\label{ilprimo}\begin{split}
P(\E) - P(\F) &\geq \frac 1M\, \H^{N-1}(\Gamma_1)-M \H^{N-1}(D\cup \Gamma_2) >\frac 1{2M}\, \H^{N-1}(\Gamma_1)\\
&\geq \frac 1{3M}\, \H^{N-1}(\partial^* I)
\geq \frac{N\omega_N^{1/N}}{3M}\, \H^N(I)^{\frac{N-1}N}
\geq \frac{N\omega_N^{1/N}}{3M}\, \bigg(\frac 1 {2M} \, |\E\Delta \F|\bigg)^{\frac{N-1}N}\,,
\end{split}\end{equation}
so that also the third requirement in~(\ref{thilvtrue}) is satisfied in this first case.\par

Let us now suppose that, instead, (\ref{firstcasenew}) is false, while~(\ref{secondcasenew}) holds, which by a simple calculation and again minding that $M\geq 1$ readily ensures that
\begin{equation}\label{conze}
\H^{N-1}(\partial^* I) \leq 5M^2 \H^{N-1}(\Gamma_2)\,.
\end{equation}
In this case, we define $\F$ by setting $F_j=E_j\cup I$, and $F_h=E_h\setminus I$ for every $h\in \{0,\, 1,\, \dots\, ,\, m\}\setminus \{j\}$. The first two requirements of~(\ref{thilvtrue}) are again clearly satisfied. Moreover, this time the perimeters of $\E$ and $\F$ coincide not only in $\R^N\setminus \partial^* I$, but also in $\Gamma_1$. Indeed, if $i=0$ then $I\cap E_0=\emptyset$, thus by construction the points of $\Gamma_1$ belong both to $\partial^* E_0$ and to $\partial^* F_0$, so their contribution to $P(\E)$ and $P(\F)$ is the same. Instead, if $i\neq 0$, then by the assumption~(\ref{0ornot0}) the intersection between $\Gamma_1$ and $\partial^* E_0$ is $\H^{N-1}$-negligible, thus also the intersection between $\Gamma_1$ and $F_0$ since points of $\Gamma_1$ belong by construction to $\partial^* F_i\cap \partial^* F_j$. And as a consequence, also if $i\neq 0$ the contribution of $\Gamma_1$ to $P(\E)$ and $P(\F)$ is the same. Therefore, the difference between $P(\E)$ and $P(\F)$ can be evaluated by noticing that $\partial^*\F\setminus\partial^*\E\subseteq D$, while $\Gamma_2\subseteq \partial^*\E\setminus\partial^*\F$, and then by~(\ref{secondcasenew}) and~(\ref{conze}) we obtain
\[\begin{split}
P(\E)-P(\F) &\geq \frac 1M\, \H^{N-1}(\Gamma_2) - M \H^{N-1}(D) >\frac 1{2M}\, \H^{N-1}(\Gamma_2)
\geq \frac 1{10M^3}\, \H^{N-1}(\partial^* I)\,,
\end{split}\]
so we conclude the third requirement of~(\ref{thilvtrue}) arguing as in~(\ref{ilprimo}). The proof is then concluded.
\end{proof}

\section{Proof of the main result\label{sec:proof}}

This section is devoted to prove our main result, Theorem~\ref{main}. We first present a simple observation.

\begin{lemma}\label{only2}
Under the assumptions of Theorem~\ref{main}, let $i\neq j \in \{0,\, 1,\, \dots\,,\, m\}$ be two indices, and let $B(x,r)$ be a ball entirely contained in $E_i\cup E_j$. Then, there exists $C>0$ and $\bar \eps>0$ such that, for any $-\bar \eps < \eps<\bar\eps$, there exists another cluster $\F$ with
\begin{equation}\label{th2}
\begin{array}{ccc}
\F\Delta \E\comp B(x,r)\,, &\qquad& B(x,r)\subseteq F_i\cup F_j\,, \\
|F_i\cap B(x,r)|=|E_i\cap B(x,r)| + \eps\,, && P(\F)\leq P(\E) + C |\eps|^\beta\,,
\end{array}
\end{equation}
where $\beta$ is given by~(\ref{defbeta}).
\end{lemma}
\begin{proof}
We can see that this is a very simple consequence of Theorem~\ref{ThPS} (that is the same as Theorem~\ref{main} with $m=1$). First of all, let us assume that $j=0$, so that the ball $B(x,r)$ only intersects $E_i$, among all the different sets $E_\ell$, $1\leq \ell \leq m$. As a consequence, the $\eps-\eps^\beta$ property for sets applied at the point $x\in \partial^* E_i$ implies the existence of two constants $\bar\eps,\, C>0$ such that, for any $-\bar\eps<\eps<\bar\eps$, there is some set $F_i$ with
\begin{align*}
F_i \Delta E_i \comp B(x,r)\,, && |F_i|=|E_i|+\eps\,, && P(F_i)\leq P(E_i)+ C |\eps|^\beta\,.
\end{align*}
Defining then the cluster $\F$ by setting $F_\ell=E_\ell$ for every $\ell \in \{1,\,2,\, \dots\, ,\, m\}\setminus \{i\}$, it is obvious that $\F$ satisfies all the requirements of~(\ref{th2}).\par

Let us instead assume that both $i,\,j\neq 0$, so that $B(x,r)$ is contained in the union $E_i\cup E_j$, and it does not intersect any $E_\ell$ with $\ell \in \{0,\, 1,\, \dots\, ,\, m\}\setminus\{i,\,j\}$. Notice that, inside $B(x,r)$, the perimeter of the cluster is given by
\[
\frac 12\,\int_{\partial^* E_i\cap B(x,r)} g(z,\nu_{E_i}(z)) \, d\H^{N-1}(z)+\frac 12\,\int_{\partial^* E_j\cap B(x,r)} g(z,\nu_{E_j}(z)) \, d\H^{N-1}(z)\,.
\]
However, $\partial^* E_i \cap B(x,r)=\partial^* E_j\cap B(x,r)$, and for $\H^{N-1}$-a.e. $z\in \partial E_i\cap B(x,r)$ one has $\nu_{E_j}(z)=-\nu_{E_i}(z)$. As a consequence, the above perimeter can be rewritten as
\[
\int_{\partial^* E_i\cap B(x,r)} \frac{g(z,\nu_{E_i}(z))+g(z,-\nu_{E_i}(z))}2 \, d\H^{N-1}(z)\,,
\]
and then inside the ball the perimeter coincides with the perimeter corresponding to the sole set $E_i$, with the auxiliary density $\tilde g(y,\nu)=\big(g(y,\nu)+g(y,-\nu)\big)/2$. Since the density $\tilde g$ satisfies the same assumptions as $g$, we can apply the $\eps-\eps^\beta$ property for sets to the set $E_i$ at the point $x$ with the density $\tilde g$. We find then two constants $\bar\eps,\, C>0$ so that, for any $-\bar\eps<\eps<\bar\eps$, there exists a set $F_i$ satisfying
\begin{align*}
F_i \Delta E_i \comp B(x,r)\,, && |F_i|=|E_i|+\eps\,, && \widetilde P(F_i)\leq \widetilde P(E_i)+ C |\eps|^\beta\,,
\end{align*}
where $\widetilde P$ is the perimeter corresponding to the density $\tilde g$. By construction, the cluster $\F$ defined by setting $F_\ell = E_\ell$ for every $\ell\in \{1,\,2,\, \dots\,,\, m\} \setminus \{i,\,j\}$ and $F_j = \big(E_j\setminus B(x,r)\big) \cup \big(B(x,r)\setminus F_i\big)$, satisfies all the requirements of~(\ref{th2}). This concludes the thesis.
\end{proof}

We are now ready to present the proof of Theorem~\ref{main}.

\proofof{Theorem~\ref{main}}
We can directly assume that $\alpha>0$ since, as already noticed at the end of Section~\ref{sec:a0}, the result for $\alpha=0$ immediately follows by Lemma~\ref{solouno} (and actually, $m(m+1)/2$ balls in Definition~\ref{defpropeb}, instead of $m(m+1)/2+1$, suffice). As already done in the proof of Lemma~\ref{solouno}, in the set of indices $i,\,j\in \{0,\, 1,\, \dots\, ,\, m\}$ we consider the relation $i\sim j$ whenever $\H^{N-1}(E_i\cap E_j)>0$ and assumption~(\ref{0ornot0}) holds up to swap $i$ and $j$. We also notice that, considering the weakest equivalence relation $\approx$ such that $i\approx j$ whenever $i\sim j$, one actually has $i\approx j$ for every $i,\,j\in \{0,\, 1,\, \dots\, ,\, m\}$.

For every pair of indices $i,\,j$ such that $i\sim j$, we select a point $x_{i,j} \in \partial^* E_i\cap \partial^* E_j$. We define now a positive constant $\bar r>0$ which is much smaller than the distance between any two of the points $x_{i,j}$, and such that for any $i\sim j$ the set $\partial^* E_i\cap \partial^* E_j$ is not entirely contained in the union of all the balls of radius $\bar r$ centered at the points $x_{i,j}$. For any $i,\,j $ with $i\sim j$, we can apply the Infiltration Lemma~\ref{InfilLem} at the point $x_{i,j}$ finding an arbitrary small $r\ll \bar r$ and a cluster $\F$ such that~(\ref{thilvtrue}) holds. There are now two possibilities; either one has $\F=\E$ for any choice of $i,\, j$ and any $r$ small enough, or there is some $i\sim j$ such that $\F\neq \E$ for every admissible $r\ll 1$.\par

Let us first assume that the second case holds, and fix some $i\sim j$ for which $\F\neq \E$ for every admissible $r\ll 1$. Let also $C$ be the constant given by the Infiltration Lemma applied at the point $x_{i,j}$. Since $\alpha>0$, in particular $g$ is continuous in the first variable. Therefore, we already know the validity of the $\eps-\eps^{\frac{N-1}N}$ property with arbitrarily small constant and with $m(m+1)/2$ balls. Therefore, up to possibly reduce the value of $\bar r$, we find some $\bar \eps>0$ such that, for any $\eps\in \R^m$ with $|\eps|<\bar\eps$, there exists a cluster $\G$ with $\E\Delta\G$ contained in the union of at most $m(m+1)/2$ balls of radius $\bar r$ not intersecting $B(x_{i,j},\bar r)$, and so that
\begin{align*}
\big|\G\big| = \big|\E\big| + \eps\,, && P(\G) \leq P(\E) + \frac C2\, |\eps|^{\frac{N-1}N}\,.
\end{align*}
By the Infiltration Lemma, we can then take a cluster $\F$ so that $\E\Delta\F\comp B(x_{i,j},\bar r)$, the properties~(\ref{thilvtrue}) hold, and $0 < |\E\Delta\F| < \bar \eps$. We can now take $\G$ so that $\E\Delta\G$ is contained in the union of at most $m(m+1)/2$ balls of radius $\bar r$ not intersecting $B(x_{i,j},\bar r)$ so that
\begin{align*}
\big|\G\big| = \big|\E\big| -|\E\Delta \F|,, && P(\G) \leq P(\E) + \frac C2\, |\E\Delta \F|^{\frac{N-1}N}\,.
\end{align*}
Defining $\E'$ the cluster which equals $\G$ outside of the ball $B(x_{i,j},\bar r)$, and $\F$ inside of this ball, by construction we have then
\begin{align*}
|\E'| = |\E|\,, && P(\E') < P(\E)\,.
\end{align*}
In other words, $\E'$ has exactly the same volume as $\E$, but strictly smaller perimeter; moreover, $\E'\Delta\E$ is contained in the union of $m(m+1)/2+1$ balls, i.e., the $m(m+1)/2$ balls given by the $\eps-\eps^{\frac{N-1}N}$ property together with $B(x_{i,j},\bar r)$. The conclusion follows then trivially in this case. Indeed, we can reduce $\bar\eps$ so to become much smaller than $P(\E)-P(\E')$. Then, for every $\eps\in\R^m$ with $|\eps|<\bar\eps$, we can easily take a cluster $\E''$, with $\E''\Delta\E'$ contained in the union of the $m(m+1)/2+1$ balls, and so that $|\E''| = |\E'|+\eps = |\E|+\eps$, and so that the difference $P(\E'')-P(\E')$ is arbitrarily small, in particular smaller than $P(\E)-P(\E')$. This is much stronger than the required $\eps-\eps^\beta$ property, since instead of having $P(\E'')\leq P(\E)+ \Cper |\eps|^\beta$ we have $P(\E'')<P(\E)$.\par

To conclude the proof, we have then now to consider the first case above. That is, for any $i\sim j$ there is some $r < \bar r$ such that the cluster $\F$ given by the Infiltration Lemma coincides with $\E$. This is equivalent to say that there exists some $r_{i,j}<\bar r$ such that $B(x_{i,j}, r_{i,j})$ is entirely contained in $E_i\cup E_j$. We can then apply Lemma~\ref{only2} to each of these balls, finding constants $C_{i,j}$ and $\bar\eps_{i,j}$ so that the result of the lemma holds. It is then clear that the $\eps-\eps^\beta$ property holds, with $\beta$ given by~(\ref{defbeta}), up to call $\bar\eps$ the minimum of all the constants $\bar\eps_{i,j}$ and $\Cper$ the maximum of all the constants $C_{i,j}$. The proof is then concluded.
\end{proof}

We conclude our paper with a discussion about the value of $\Cper[t]$ when $\alpha=0$ and $g$ is continuous in the first variable, and with a quick sketch of the proof of Corollary~\ref{toadd} (which is quite standard, for instance it is very similar to the one made in~\cite{PS19} for the case of sets).

\begin{remark}\label{bestC}
Let us consider the case $\alpha=0$ when the function $g$ is continuous in the first variable. In this case, we know the validity of the $\eps-\eps^{\frac{N-1}N}$ property with an arbitrarily small constant $\Cper$. As discussed in Remark~\ref{remCsmall}, this means that one is able to assume $\Cper$ as small as one needs, up to select a sufficiently small $\bar\eps$. This is a consequence of the fact that, in Lemma~\ref{lemmaa=0}, the constant $K$ can be any constant larger than $C\sqrt[N]{\omega_{\bar x}}$, and $\omega_x=0$ for every $x\in\R^N$ since $g$ is continuous in the first variable. However, sometimes it is important to have a more precise estimate of how small can $\Cper$ be taken. More precisely, instead of changing the value of $\bar\eps$, one calls $\Cper[t]$ the best constant that one can take in~(\ref{defepsbeta}) when $|\eps|\leq t$. And then, instead of just observing that $\Cper[t]\searrow 0$ when $t\searrow 0$, one wants to estimate the rate of convergence.\par
This is easily done by checking in detail the proof of Lemma~\ref{lemmaa=0}, in particular the calculations of Step~VIII. Indeed, let us take a large closed ball $A$ such that, for any $0\leq i<j\leq m$ with $\H^{N-1}(\partial^* E_i\cap\partial^* E_j)>0$ one has also $\H^{N-1}(\partial^* E_i\cap\partial^* E_j\cap A)>0$, and call $M=\max \{M_x,\, x\in A\}$ and $\omega(t)=\max\{ \omega_x(t),\, x\in A\}$. Thanks to Lemma~\ref{lemmaa=0}, we get the existence of a constant $C=C(M,N,m)$ such that $\Cper[t]$ is smaller than $C \sqrt[N]{\omega(a \sqrt N/2)}$ if we are able to modify volumes up to $t$ inside cubes of side $a$. We have then to check the relation between $a$ and $t$ (which corresponds to what was called $\bar\eps$ in the proof).\par

We observe that according to~(\ref{bestL}) we chose $L = \omega_x(a)^{-1/N}$, and by~(\ref{bestrho}) we have $\rho L^{N-1}\approx\sqrt[N]{\omega_x(a)}$, which implies $\rho\approx \omega_x(a)$. And finally, (\ref{qualeps}) tells us that
\[
t\approx a^N \omega_x(a)/\rho^N \approx a^N \omega_x(a)^{1-N}\,.
\]
Summarizing we have seen that, up to multiplicative constants only depending on $M,\, N$ and $m$,
\[
\Cper[a^N \omega(a)^{1-N}] \lesssim \omega(a \sqrt N)^{1/N}\,.
\]
Since $\omega(t)\searrow 0$ for $t\searrow 0$, for small $t$ we have then
\[
\Cper[t^N] \leq \Cper[t^N  \omega(t/\sqrt N)^{1-N}/N^{N/2}] = \Cper[(t/\sqrt N)^N  \omega(t/\sqrt N)^{1-N}] \lesssim \omega(t)^{1/N}\,,
\]
which can be rewritten in the more useful form
\[
\Cper[t] \lesssim \sqrt[N]{\omega(t^{1/N})}\,.
\]
\end{remark}

\proofof{Corollary~\ref{toadd}}
Assume that both $f$ and $g$ are globally bounded from above and below, and that the $\eps-\eps^{\frac{N-1}N}$ property with arbitrarily small constant holds. Then, the boundedness of the minimal cluster is very easy to get. Indeed, call $\E_t=\E\cap B_t$ the intersection of the cluster $\E$ with the ball $B_t$ centered at the origin and with radius $t$, and call $v(t)=\big| | \E \setminus \E_t|\big|$ the total volume of the remaining part. The boundedness of $f$ and $g$ ensures that
\[
P(\E\setminus\E_t) \geq C v(t)^{\frac{N-1}N}
\]
for a suitable constant $C$. For $t$ sufficiently large, the $\eps-\eps^{\frac{N-1}N}$ property with small constant allows to modify the cluster $\E_t$ so to get some cluster $\widetilde\E$ with the same volume as $\E$, with
\[
P(\widetilde\E)-P(\E_t) \ll v(t)^{\frac{N-1}N}\,.
\]
Moreover, again the boundedness of $f$ and $g$ implies that
\[
|v'(t)| \approx \H^{N-1}\big(\E\cap \partial B_t\big)\approx P(\E_t) +P(\E\setminus\E_t)-P(\E)\,.
\]
The minimality of $\E$ ensures $P(\E)\leq P(\widetilde\E)$, which by the above estimates give
\[
|v'(t)| \geq C' v(t)^{\frac{N-1}N}
\]
for some constant $C'$. This implies that $v(t)=0$ for $t$ large enough, that is, the cluster is bounded. Hence, the proof of the corollary in the first case is completed (notice that the validity of the $\eps-\eps^\beta$ property with $\beta>(N-1)/N$ implies the validity of the $\eps-\eps^{\frac{N-1}N}$ property with arbitrarily small constant).\par

In the second case, we can easily reduce to the first one. More precisely, assume by contradiction that the cluster $\E$ is not bounded, and in particular (up to renumbering) assume that the unbounded chambers are $E_1,\, E_2,\, \dots\, ,\, E_k$ for some $k\leq m$. The fact that $g$ has a limit at infinity implies that the obscillation of $g$ becomes arbitrarily small for points sufficiently far from the origin. As just noticed in Remark~\ref{bestC}, a small bound on the obscillation allows to get a small constant in Lemma~\ref{lemmaa=0}. Therefore, we can get arbitrarily small constant in Lemma~\ref{lemmaa=0} for points which are in the boundary of the chambers $E_i$ with $1\leq i \leq k$. Because of the possible existence of bounded chambers $E_i$ with $i>k$, this is not sufficient to say that the $\eps-\eps^{\frac{N-1}N}$ property holds with an arbitrarily small constant; however, this is true \emph{for small volumes $\eps\in\R^m$ such that $\eps_i=0$ for every $i>k$}. And in turn, this weaker version of the property is exactly what we need; indeed, since the bounded chambers $E_i$ with $i>k$ are finitely many, for $t$ large enough the cluster $\E\setminus \E_t$ may contain non-empty chambers only for indices $1\leq i\leq k$. Therefore, the argument above for the first case applies also in this second case.
\end{proof}

\end{document}